\documentclass[12pt]{amsart}

\usepackage{amssymb,amsbsy,amsmath,amsfonts,amssymb,amscd}
\usepackage{latexsym}
\usepackage{stmaryrd}
\usepackage{graphics}
\usepackage{color}
\usepackage{marginnote}

\input xy
\xyoption{all}

\usepackage{tikz}
\usepackage{tikz-cd}
\usetikzlibrary{decorations.markings}
\tikzset{%
  show curve controls/.style={
    postaction={
      decoration={
        show path construction,
        curveto code={
          \draw [blue] 
            (\tikzinputsegmentfirst) -- (\tikzinputsegmentsupporta)
            (\tikzinputsegmentlast) -- (\tikzinputsegmentsupportb);
          \fill [red, opacity=0.5] 
            (\tikzinputsegmentsupporta) circle [radius=.5ex]
            (\tikzinputsegmentsupportb) circle [radius=.5ex];
        }
      },
      decorate
}}}

\usetikzlibrary{arrows}
\usepgflibrary{shapes}
\usepackage{latexsym}
\usepackage{graphics}
\usepackage{amsmath,latexsym,amssymb}
\usepackage[german,english]{babel}
\input xy
\xyoption{all}

\newcommand\sO{{\mathcal O}}

\newcommand\sB{{\mathcal B}}

\newcommand\sK{{\mathcal K}}

\newcommand\Ga{\Gamma}

\newcommand\ga{\gamma}

\DeclareMathOperator{\tor}{Tor}

\def\Bbb{\bf}
\newcommand{\CC}{\ensuremath{\mathbb{C}}}

\newcommand{\ZZ}{\ensuremath{\mathbb{Z}}}
\newcommand{\QQ}{\ensuremath{\mathbb{Q}}}

\newcommand{\ra}{\ensuremath{\rightarrow}}

\def\eea{\end{eqnarray*}}
\def\bea{\begin{eqnarray*}}

\newcommand\dual{\mathrel{\raise3pt\hbox{$\underline{\mathrm{\thinspace d
\thinspace}}$}}}

\newcommand\qe{\ifhmode\unskip\nobreak\fi\quad $\Box$}

\def\BOX{\hfill\lower.5\baselineskip\hbox{$\Box$}}

\newtheorem{theo}{Theorem}[section]
\newtheorem{remarkk}[theo]{Remark}
\newenvironment{rem}{\begin{remarkk}\rm}{\end{remarkk}}

\newtheorem{defin}[theo]{Definition}

\newtheorem{prop}[theo] {Proposition}

\newtheorem{lemma}[theo]{Lemma}
\newtheorem{example}[theo]{Example}
\newenvironment{ex}{\begin{example}\rm}{\end{example}}

\newtheorem{notation}[theo]{Notation}

\newcommand{\bN}{{\mathbb N}}

\renewcommand{\a}{\alpha}
\renewcommand{\b}{\beta}

\newcommand{\e}{\varepsilon}

\newcommand{\x}{\xi}

\newcommand{\cutoff}[1]{}

\DeclareMathOperator{\Aut}{Aut}

\newcommand{\tch}{{T\!C\!h}} 

\begin{document}

\title[Stabilization for the Homology of moduli of curves with  symmetries]
{Genus stabilization for the Homology  of  moduli spaces 
of orbit-framed curves with symmetries-I}
\author{Fabrizio Catanese, Michael L\"onne, Fabio Perroni}


\date{\today}

\maketitle

\begin{abstract}

In a previous paper, \cite{CLP16}, we showed   that the moduli space of curves $C$ with a $G$-symmetry (that is,  with  a faithful 
action of a finite group $G$),   having  a fixed  generalized homological invariant $\e$, is irreducible 
 if the genus $g'$ of  the quotient curve $C' : = C/G$ satisfies 
 $g'>>0$.

Interpreting this result as  stabilization  for the 0-th homology group of the moduli space of curves with $G$-symmetry,
we begin here a program for showing  genus  stabilization for all the homology groups of these spaces, in similarity  to the
results of   Harer \cite{harer85} for the moduli space of curves.

In this first paper we prove homology stabilization for  a variant of the moduli space where one  $G$-orbit is tangentially  framed.

\end{abstract}


\section{Introduction}

The main purpose of this article is to prove  a first result concerning  the stabilization of the homology of  
the moduli spaces of curves admitting a given  symmetry group $G$.

Recall that Harer \cite{harer85} proved that the homology groups of the moduli space of curves $\mathfrak M_g$ 
stabilize for $g >>0$,  indeed, more precisely,
$$ H_k( \mathfrak M_g, \QQ) \cong  H_k( \mathfrak M_{g+1}, \QQ), \  {\rm for} \ g \geq 3k+1.$$ 

Fix $
\Sigma = \Sigma_g$ to be a fixed  compact oriented differentiable manifold of real dimension $2$ and genus $g$: 
letting $\mathcal{C}S(\Sigma)$ be the space of complex structures on $\Sigma$  compatible with the given orientation,
then  the moduli space of curves
$\mathfrak{M}_g $, and  the Teichm\"uller space $\mathcal{T}_g $ are  the quotients:

$$
\mathfrak{M}_g := \mathcal{C}S(\Sigma)/{\rm Diff}^+(\Sigma), \ 
\mathcal{T}_g := \mathcal{C}S(\Sigma)/{\rm Diff}^0(\Sigma) \ 
\Rightarrow \mathfrak{M}_g = \mathcal{T}_g/ {\rm Map}_g = \mathcal{T}_g/ \Ga_g
$$
 where $\Ga_g : = {\rm Map}_g = \pi_0 ( {\rm Diff}^+(\Sigma))=  {\rm Diff}^+(\Sigma) / {\rm Diff}^0(\Sigma) $.

Teichm\"uller's theorem says that $\mathcal{T}_g \subset \CC^{3g-3}$ is an open subset diffeomorphic to a ball. 
Moreover $\Ga_g = {\rm Map}_g $ acts properly discontinuously on $\mathcal{T}_g$, but not freely.

Hence  the rational cohomology (resp.  homology) of the moduli space is calculated by group cohomology (resp. group homology):
$$
H^*(\mathfrak{M}_g, \QQ) \cong H^*(\Ga_g, \QQ) , H_*(\mathfrak{M}_g, \QQ) \cong H_*(\Ga_g, \QQ) .
$$

Indeed, Harer proved some stronger results concerning the integral homology of the groups $\Ga^n_{g,r}$, mapping class
groups of the groups  ${\rm Diff}^+(\Sigma_{g,r}^n)$ of diffeomorphisms of a Riemann surface  $\Sigma = \Sigma_{g,r}^n$ with $n$-punctures and $r$ boundary components,
acting as the identity on the boundary curves  and on the punctures: he showed  that, for $g$ large, these are independent of $g,r$ (to consider also the punctures and the boundary components was necessary to pursue several  inductive steps).

Inside the moduli space of curves there are some important special  subvarieties, corresponding to curves
admitting a $G$-symmetry, namely such that there is an injection $ G \hookrightarrow \Aut(C)$, where $G$ is
a fixed finite group (of order $\leq 84(g-1)$ by Hurwitz' theorem, if, as we  assume here  throughout,  $g\geq 2$).

These subvarieties are the union of locally closed sets which are better understood  as the image of a $G$-marked moduli space $\mathfrak{M}_g(G)$,
whose irreducible components are determined by the type of the topological action of $G$ on $C$.

The topological type (and the differentiable type)  of the action is completely 
determined by the  action of $G$ on the fundamental group $\Pi_g$ of $C$, up to inner 
automorphisms.
Hence we get an injective  group homomorphism
$$
\rho \colon G  \hookrightarrow {{\rm Aut}(\Pi_g)}/{{\rm Inn}(\Pi_g)}={\rm Out}(\Pi_g) \, 
$$
with image in  an index $2$ subgroup:
$$\rho \colon G  \hookrightarrow {\rm Out}^+(\Pi_g)=\Ga_g \cong {{\rm Diff}^+(C)}/{{\rm Diff}^0(C)}.$$

It turns out that curves with a fixed topological  type of action $\rho$ correspond to the fixed locus
$
\mathcal{T}_{g,\rho} := \mathcal{T}_g^{\rho(G)}
$ of $\rho(G)$ in $\mathcal{T}_g$, which (see \cite{FabIso}) is diffeomorphic to a ball: in particular the
rational cohomology of these components $\mathfrak{M}_{g,\rho} $ is then calculated  via group cohomology
of the normalizer of $\rho(G)$ inside the mapping class group.

In practice, however, these components $\mathfrak{M}_{g,\rho} \subset \mathfrak{M}_{g}(G) $
are better   understood through the classical study of the quotient map
$$ C \ra C' : = C/G.$$ 
 We let  $g':=g(C')$ be  the  genus of  the quotient curve  $C'$,   $\sB=\{ y_1, \ldots , y_n\}$ be the branch locus, and 
let $m_i$  be the  branching multiplicity
of $y_i$  (we assume $2 \leq  m_1 \leq m_2 \leq \dots \leq m_n$).

These  numbers $g', n, m_1, \ldots , m_n$ are constant in each  
irreducible variety $\mathfrak M_{g,\rho}$, they form the `primary numerical type' of the cover
$C\to C'$, and 
$C\to C'$ is determined (by virtue of Riemann's existence theorem)  by the monodromy:
$
\mu \colon \pi_1(C'\setminus \sB, y_0) \to G \, .
$

 After the choice of a diffeomorphism of the pair $(C', \sB)$ with $(\Sigma_{g'}^n, \sB_\Sigma)$,  we have:

\begin{small}
$$ \pi_1(C' \setminus \sB, y_0) \cong \Pi_{g', n}:= \langle  \a_1, \b_1 , \ldots , \a_{g'}, \b_{g'}, \ga_1, \ldots , \ga_n \mid \prod_{j=1}^{g'} [\a_j, \b_j] \prod_{i=1}^n \ga_i =1 \rangle ,$$
\end{small}

hence,  after the choice of a basis for the fundamental group of $\Sigma_{g'}^n  \setminus \sB_\Sigma$, the datum of $\mu $ is equivalent to the datum of a Hurwitz -vector

$$
v : = ( a_1, b_1, \ldots , a_{g'}, b_{g'}, c_1, \ldots , c_n ) \in G^{2g'+n}
$$

s.t.  
(1) $G$ is generated by  

 $ a_1, b_1, \ldots , a_{g'}, b_{g'}, c_1, \ldots , c_n $,
 
($c_i \not= 1_G$ $\forall i$,  since $c_i$ has order $m_i \geq 2$),   and

(2)  $\prod_{j=1}^{g'}[a_j,b_j]\prod_{i=1}^n c_i =1.$

%
%

 The upshot is that the number of the 
components of $\mathfrak{M}_g(G)$  is  independent of $g'$ if the genus $g'$ of the quotient curve is large enough
 and the monodromy `stays the same': 
for this purpose,  as in Harer's theorem,
one  needs to define   genus stabilization: 
this is done   considering  Riemann surfaces obtained adding (compatibly)  a handle to  $C' , \Sigma_{g'}$,
and  extending the monodromy in a trivial way to the handle,  namely, setting:
$
\mu(\a_{g'+1})= \mu(\b_{g'+1})=1 \in G  
$
 (see the pictures in subsection 2.1).

 Using the notion of genus stabilization a breakthrough was made by Nathan Dunfield and William Thurston 
in the \'etale case where the action of $G$ is free: 

\begin{theo}[Dunfield-Thurston]
In the unramified case ($n=0$), for $g' >> 0$, the  unmarked topological types
are in bijection with the set
$$
{H_2(G,\ZZ)}/{{\rm Aut}(G)} \, .
$$
\end{theo}

The second homology of $G$ appears here naturally, since the monodromy $\mu : \pi_1(C') \ra G$
yields a corresponding continuous map of classifying spaces $f : C' \ra B(G)$, and
the image of the fundamental class of $C'$ via $H_2(f)$ yields  the homology 
invariant $\e \in H_2(G,\ZZ)$.

In \cite{CLP16} we extended the above theorem to the ramified case with a technically complicated
procedure:

\begin{theo}\label{stablebranched}
For $g'>>0$, the unmarked topological types
 are in bijection with the set of admissible classes   via the map $\hat{\e}$ associated to the generalized homology  invariant $\e$
 introduced in \cite{CLP15}.
\end{theo}

Interpreting the previous result as a stabilization result for the group $H_0 (\mathfrak{M}_g(G), \ZZ)$
it was natural to ask for an extension of this result to higher homology groups, as done by Harer for
$\mathfrak{M}_g$.

For the  purpose of explaining how this goes, we want,  following Hain and Looijenga \cite{h-l}, to spell out some other interpretation
of the groups $\Ga^n_{g,r}$ as mapping class groups.

Again, fix an oriented differentiable  Riemann surface $\Sigma_{g'}$ with $n$ fixed points 
$y_1, \ldots, y_n$ and $r$ 
fixed points  $z_1, \ldots , z_r$ 
(the $n+r$ points are all distinct) and a fixed tangent vector $v_j \neq 0 $ at each  $z_j$. We can 
 take the quotient of  the space of complex structures on $\Sigma$ by different
groups

$$  {\rm Diff} (\Sigma)^n_r \subset {\rm Diff} (\Sigma)^{n+r}  \subset  {\rm Diff} (\Sigma)$$
where the first subgroup fixes all the points $y_i, z_j$ 
individually and  acts as the identity on the tangent space
at each point  $z_j$, the second just fixes all the points  $y_i, z_j$.

The resulting moduli spaces are then 
$$ \mathfrak{M}_{g',r}^n  \ra \mathfrak{M}_{g'}^{r+n} \ra 
\mathfrak{M}_{g'},$$
where the first one is the moduli space of complex curves $C'$ with $n$ marked points and $r$ marked (nonzero)
tangent vectors, the second is the moduli space of complex curves $C'$ with $n+r$ marked points. 

The corresponding mapping class groups are 
$$\Ga_{g',r}^n \ra  \Ga_{g'}^{n+r} \ra \Ga_{g'} ,$$
since to a Riemann surface $\Sigma'$ with $r$ boundary components we can attach a disk along each boundary circle,
and if we have a diffeomorphism acting as the identity (in a neighbourhood of) on the boundary of $\Sigma'$ we 
can extend it  to the disk as the identity.

  To compare the picture  downstairs  on $C'$
to the picture upstairs on $C$, we observe that the choice of a tangent vector $v_j$ at  $z_j$ 
is equivalent to
a tangent framing of the $G$-orbit $\sO_j$ lying above  $z_j$: 
this means that we give $|G|$ tangent vectors
at the points of $\sO_j$ which are permuted simply transitively by $G$, 
and we call this a {\it tangential orbit framing}.

Again we have  corresponding Teichm\"uller spaces, hence calculating the rational cohomology groups 
of our  moduli spaces $\mathfrak{M}_g (G)$ boils down to calculations related with
some  cohomology (homology)  of groups related to  the  mapping class groups.

The road  to do this was paved by  two important  papers, the first one by Ellenberg, Venkatesh and Westerland
\cite{e-v-w} concerning Hurwitz spaces, which has been  an important source of inspiration for us  and, as pointed out by the referee,  for many other authors
who studied Hurwitz spaces, see \cite{homstab} for a good list of references.

  Our contribution in this direction is to introduce instead some different (noncommutative) rings of connected components, and we consider certain 
spectral sequences associated to a Koszul type complex. We need to extend many of the
homological algebra results of 
\cite{e-v-w}, since we work  here with $\ZZ$-coefficients instead of coefficients taken   in a given field.

On the other hand, it is important to observe that there are at least two kinds of stabilization procedures,
the first one being {\bf branch stabilization}, see for instance \cite{lonne}, 
and the second one being {\bf genus stabilization}, the one to which we are particularly interested in this article: and our ring of connected components is different from the previous ones, as dictated by our goal of studying genus stabilization.

For genus stabilization, in particular 

for some crucial results concerning the degeneration of a  first page spectral sequence,
the complexes  used by Harer are no longer suitably sufficient.

Hence the second  important paper for us has been the one  by Hatcher and Vogtmann \cite{h-v},
and we have 
successfully   used  the complex 
of tethered chains,  introduced by them.

This is our main result

\begin{theo}\label{main}
Let $\mathfrak{M}_{g}(G)^*$ be the moduli space of curves $C$ with a $G$-action 
and with one  tangentially-framed $G$- orbit $\sO$.

Then, for $g > > 0$, the homology group $H_i(\mathfrak{M}_{g}(G)^*, \QQ)$ 
is independent of $g$.

\end{theo}

We could  of course establish a more general theorem taking $r$ tangentially-framed $G$- orbits $\sO_1, \dots, \sO_r$,
but the main line of argument, while remaining   entirely similar, needs a quite heavy notation.

For the sake of simplicity of notation, we carry out the proof of Theorem \ref{main} in the case where 
the action of $G$ is free outside
of $\sO$, that is, the number  $n=0$. 

In the last Section \ref{general} we develop the arguments needed for the more general  case where there are more branch points, and this proves the full statement  Theorem \ref{main}. 

Observe moreover that, if in Theorem \ref{main}  $n=0$ and  we restrict to the subspace of the moduli space 
where the orbit $\sO$ is in bijection with $G$ (i.e.,  the local monodromy at  $z_1$ is the identity),
we obtain an unramifed covering $C \ra C'$, hence  our moduli space maps onto $\mathfrak{M}_{g}(G)^{unr}$.

We plan, in a future paper, to show homology stabilization also for the space 
$\mathfrak{M}_{g}(G)^{unr}$ (considered in \cite{DT}) and, in  general, for the 
more general  space $\mathfrak{M}_{g}(G)$.

Our main Theorem \ref{main} is a consequence of a more technical result, which we would now like to illustrate.

First of all, we start from the Hurwitz vectors, which describe the monodromy of the covering $C \ra C' = C/G$.

 Since   in Theorem \ref{main} we have $r=1$,  in the case   $n=0$ the set of Hurwitz vectors is in bijection with  the set \[
 G^{2g'} = \left\{ \: (a_1,b_1, a_2,b_2, \dots, a_{g'}, b_{g'}) \: \big |\: a_i, b_i\in G \:
\right\}.
\]
 The mapping class group $ \Ga_{g', 1}:=\Ga^0_{g', 1}$ acts on them using the identification 
of $\operatorname{Hom}(\pi_1(\Sigma',y_0), G)$
with $G^{2{g'}}$ by means of a  fixed geometric basis.

Let us denote by $R_{g'}:= \ZZ \langle G^{2{g'}}/\Gamma_{g', 1} \rangle$, for $g'>0$, and  set $R_0:= \ZZ$.

 In Section \ref{ring of connected components} 
we show the existence of natural  homomorphisms yielding a non-commutative graded ring structure on 
$$
R:= \oplus_{g' \geq 0} R_{g'} \, .
$$
This graded ring is called the {\bf ring of connected components}, it depends heavily on the group $G$,
and we consider (graded) left $R$-modules, establishing analogues of some homological algebra results
of \cite{e-v-w}, which we have to prove in detail since we are dealing with coefficients in $\ZZ$
and not  in a field.

Our  typical example is the  $q$-th homology-module 

\begin{equation}\label{Module}
M(q)= \oplus_{g'\geq 0}M(q)_{g'} :=  \oplus_{g'\geq 0} H_q ( \Gamma_{g', 1}, \ZZ \langle G^{2{g'}} \rangle ) \, .
\end{equation}
Note that $M(0)=R$.

For every $\ell\geq 0$
and $(\tilde{a}_1, \tilde{b}_1, \ldots , \tilde{a}_\ell, \tilde{b}_\ell) \in G^{2 \ell}$, the map
\begin{eqnarray*}
\ZZ \langle G^{2g'} \rangle &\to& \ZZ \langle G^{2\ell + 2g' } \rangle \\
(a_1, b_1,  \ldots , a_{g'}, b_{g'}) &\mapsto& 
(\tilde{a}_1, \tilde{b}_1, \ldots , \tilde{a}_\ell, \tilde{b}_\ell, a_1, b_1,  \ldots , a_{g'},  b_{g'})
\end{eqnarray*}
is equivariant with respect to the inclusion $ \Gamma_{g', 1} \hookrightarrow \Gamma_{\ell+g', 1}$,
and, considering  the class { $\llbracket \tilde{a}_1, \tilde{b}_1, \ldots , \tilde{a}_\ell, \tilde{b}_\ell \rrbracket$}
of  $(\tilde{a}_1, \tilde{b}_1, \ldots , \tilde{a}_\ell, \tilde{b}_\ell)$ inside $R_\ell$,  it induces a homomorphism 
\begin{eqnarray*}
 H_q ( \Gamma_{g', 1}, \ZZ \langle G^{2g'} \rangle ) & \to & H_q ( \Gamma_{\ell+g', 1}, \ZZ \langle G^{2\ell + 2g' } \rangle ) \\
 m &\mapsto& \llbracket \tilde{a}_1, \tilde{b}_1, \ldots , \tilde{a}_\ell, \tilde{b}_\ell\rrbracket  m \, 
\end{eqnarray*}
hence  it defines   a structure of graded left $R$-module on $M(q)$.

\bigskip

The algebraic analogue of  genus stabilization is now expressed by an operator $U$, such that $U(m)=\llbracket 1,1\rrbracket m$,
which is the main actor on the stage.

\bigskip

Then the main homological argument is concerned with bounding,
via several exact sequences, the degrees of certain graded modules, especially the Kernel and Cokernel 
of multiplication by $U$: here the   \textit{degree} of a graded module $M$ is defined as  
${\rm deg}(M): = \sup \{ g' \, | \, M_{g'}\not= 0\}$.

\smallskip

With the above notation, our technical result (Theorem \ref{Ustabilizes})
is 
\begin{theo}\label{Ustabilizes}
Let $G$ be a finite group. Let $R$ be the ring of connected components associated to $G$,
and let $M(q)$ be the graded left $R$-module  defined in \eqref{Module}. 

Then there are homological constants  $\tilde{A}(R)$ and $A(R)$ such that 
 the homomorphism
$$
{U} \colon  M(q)_{g'} \to M(q)_{g'+1}
$$
is an isomorphism for every  $g'\geq (8\tilde{A}(R) + \deg (U))q + \tilde{A}(R) +6A(R) +2$. 
\end{theo}

The statement of  the above theorem is exactly the stabilization statement for homology. 

\bigskip

In section two we recall known results on the simplicial complex of tethered chains introduced by 
\cite{h-v} and deal with the basic spectral sequences on which our results are relying.

In section three we  
generalize the theory of Koszul like complexes introduced by \cite{e-v-w}
considering $\ZZ$-modules instead of vector spaces, and applying the theory to the
ring of connected components of the moduli spaces of curves with $G$-symmetry.
The investigation of the homological algebra of this ring and related results are
our key tools.

Section four contains the main result.
\bigskip

\subsection{Notation}
Modules are understood to be left modules, unless the contrary is explicitly specified. 

Similarly for group actions, since every left action is equivalent to a right action, via the formula
$$ x g : = g^{-1} x,$$
in the case where there is only one natural action of the group $G$ on a set $S$,
we generally consider  left actions.
However,  we shall  use our customary notation  $S/G$ instead of $G \backslash S$ for the
orbit set, even if $G$ is acting on the left. 

One exception to this notation concerns the case where there are two different actions
of $G$ on $S$.  For instance the notation for the set of cosets of a subgroup $H$ in a group $G$
is the self explanatory one:
 $G/H$ denotes the set of left cosets of $H$, $\{ gH \, | \, g\in G \}$,
while $H\backslash G = \{ Hg \, | \, g\in G\}$ denotes the set of right cosets of $H$.

For any group and any pair of elements $x, y$ in it, we
define conjugation through the formula $x^y := y^{-1}xy$.


\section{Chain complexes and spectral sequences}\label{ccss}

A \textit{Hurwitz vector of type} $g'$ is a tuple  of the following form,
\[
 (a_1,b_1, a_2,b_2, \dots, a_{g'}, b_{g'}) \quad \mbox{where} \quad a_i, b_i \in G \, , 
\]
therefore  $G^{2g'}$ is also regarded as the set of such  vectors.
Let us  denote, from now on, $\Sigma^0_{g', 1}$ by $\Sigma_{g', 1}$ and let us 
fix a geometric basis $\{ \alpha_1, \beta_1, \ldots , \alpha_{g'} , \beta_{g'}\}$ 
of $\pi_1( \Sigma_{g', 1}, y_0)$, with $y_0 \in \partial  \Sigma_{g', 1}$. 
Let us identify $\operatorname{Hom}(\pi_1( \Sigma_{g', 1},y_0), G)$ with $G^{2g'}$ as usual:
$\mu \in \operatorname{Hom}(\pi_1( \Sigma_{g', 1},y_0), G)$ corresponds to 
\[
(a_1,b_1,  \dots , a_{g'}, b_{g'}) := (\mu (\alpha_1), \mu(\beta_1), \ldots , \mu (\alpha_{g'}), \mu (\beta_{g'})) \, .
\]
The mapping class group  $\Gamma_{g', 1}:= \Gamma^0_{g', 1}$ 
acts naturally on $\pi_1( \Sigma_{g', 1}, y_0)$, and hence on $G^{2g'}$,
by means of the previous identification. More precisely, for $\gamma \in  \Gamma_{g', 1}$ and 
$\mu \in \operatorname{Hom}(\pi_1( \Sigma_{g', 1},y_0), G)$, 
\[
\gamma \cdot \mu :=   \mu \circ \pi_1(\ga^{-1}).
\]
Notice that this is a left-action.

\begin{rem}

In general, one  may replace the fundamental group $\pi_1(X, y_0)$ with the fundamental group $\pi_1(X, I)$ relative
to a closed, connected and simply connected subspace $I \subset X$ containing $y_0$.

In our situation, to avoid cumbersome and useless verifications of associativity of glueing handles,
we replace  $\pi_1(\Sigma_{g', 1}, y_0)$ by $\pi_1(\Sigma_{g', 1}, I_{\partial})$,
where  $I_{\partial}$ is a segment on the boundary. This segment is coloured in red in the Figures \ref{tethered} and \ref{simplex}.
\end{rem}

\subsection{The simplicial complex of tethered chains}\label{tch complex}
We recall here some definitions from \cite{h-v}.
A \textit{chain} in $ \Sigma_{g', 1}$ is an ordered pair $(\hat{\alpha}, \hat{\beta})$ of simple closed curves
intersecting transversely in a single point, together with an orientation on $\hat{\beta}$.  
A \textit{tethered chain} in $ \Sigma_{g', 1}$ is a pair $\left((\hat{\alpha}, \hat{\beta}), t \right)$, where $(\hat{\alpha}, \hat{\beta})$
 is a chain 
and $t$ is an arc connecting the positive side of $\hat{\beta}$ with a point in some finite collection $P$ of disjoint open intervals
(not circles) in $\partial  \Sigma_{g', 1}$, such that the interior of $t$ is disjoint from $\hat{\alpha}$, 
$\hat{\beta}$ and $\partial  \Sigma_{g', 1}$.
In our situation $P$  always consists of one interval containing the base point $y_0$.

The \textit{complex of tethered chains on  $ \Sigma_{g', 1}$}, $ \tch_{g', 1}$, is the simplicial complex 
whose vertices correspond to isotopy classes of tethered chains, and a set of vertices spans a simplex 
if the tethered chains can be chosen to be disjoint.

\bigskip

\begin{figure}

\begin{tikzpicture}[scale=0.85]


\foreach \x in {0,5,9,13}
{
\draw [very thick, dashed] 
 (-1.6+\x,2.5)  .. controls ++(85:-1) and ++(0:.1) .. (-2+\x,1);
\draw [very thick] 
 (-1.6+\x,2.51)  .. controls ++(85:-.2) and ++(82:.2) .. (-1.66+\x,2);
\draw [very thick] 
 (-2+\x,1) .. controls ++(0:-.1) and ++(-75:1) .. (-2.5+\x,2);
 

\draw [very thick] (-1.6+\x,2.5) -- (2.4+\x,2.5);

\draw [very thick] 
 (-2+\x,1)  .. controls ++(-5:2) and ++(0:-2.7) .. (0+\x,-2.9)
  .. controls ++(0:2.7) and ++(-5:-2) .. (2+\x,1);

\draw [very thick] 
 (-0.1+\x,0.1) .. controls ++(0:.2) and ++(90:.7) .. (.5+\x,-1)
 .. controls ++(90:-.7) and ++(0:.2) .. (-0.1+\x,-2.1);

\draw [very thick] 
 (0+\x,0.07) .. controls ++(0:-.2) and ++(90:.7) .. (-.5+\x,-1)
 .. controls ++(90:-.7) and ++(0:-.2) .. (0+\x,-2.07);
 
\draw [blue] (0+\x,-.9) ellipse (.9 and 1.6);

\draw [blue]   
 (0.4165+\x,-0.5) .. controls ++(108:-.2) and ++(100:-.2) .. (1.22+\x,-.3);
\draw [blue, dashed]   
 (0.4165+\x,-0.5) .. controls ++(108:.2) and ++(100:.2) .. (1.22+\x,-.3);

\draw [blue]   
 (-0.4+\x,2) .. controls ++(100:-.2) and ++(120:.2) .. (0+\x,.7);

\draw [red, very thick] (-2.51+\x,2) -- (1.51+\x,2);
};

\foreach \x in {0,13}
\draw [very thick] 
 (2.4+\x,2.51)  .. controls ++(85:-1) and ++(0:.1) .. (2+\x,1)
  .. controls ++(0:-.1) and ++(-75:1) .. (1.5+\x,2);

\end{tikzpicture}
\caption{Attaching one handle endowed with a tethered chain.
The base point for the fundamental group is any point on the red segment,
equivalently, we take paths with endpoints in the segment.}\label{tethered}
\end{figure}
\bigskip

\begin{figure}

\begin{tikzpicture}[scale=0.85]


\foreach \x in {0,4,8,12}
{
\draw [very thick, dashed] 
 (-1.6+\x,2.5)  .. controls ++(85:-1) and ++(0:.1) .. (-2+\x,1);
\draw [very thick] 
 (-1.6+\x,2.51)  .. controls ++(85:-.2) and ++(82:.2) .. (-1.66+\x,2);
\draw [very thick] 
 (-2+\x,1) .. controls ++(0:-.1) and ++(-75:1) .. (-2.5+\x,2);
 

\draw [very thick] (-1.6+\x,2.5) -- (2.4+\x,2.5);

\draw [very thick] 
 (-2+\x,1)  .. controls ++(-5:2) and ++(0:-2.7) .. (0+\x,-2.9)
  .. controls ++(0:2.7) and ++(-5:-2) .. (2+\x,1);

\draw [very thick] 
 (-0.1+\x,0.1) .. controls ++(0:.2) and ++(90:.7) .. (.5+\x,-1)
 .. controls ++(90:-.7) and ++(0:.2) .. (-0.1+\x,-2.1);

\draw [very thick] 
 (0+\x,0.07) .. controls ++(0:-.2) and ++(90:.7) .. (-.5+\x,-1)
 .. controls ++(90:-.7) and ++(0:-.2) .. (0+\x,-2.07);
 
\draw [blue] (0+\x,-.9) ellipse (.9 and 1.6);

\draw [blue]   
 (0.4165+\x,-0.5) .. controls ++(108:-.2) and ++(100:-.2) .. (1.22+\x,-.3);
\draw [blue, dashed]   
 (0.4165+\x,-0.5) .. controls ++(108:.2) and ++(100:.2) .. (1.22+\x,-.3);

\draw [blue]   
 (-0.4+\x,2) .. controls ++(100:-.2) and ++(120:.2) .. (0+\x,.7);

\draw [red, very thick] (-2.51+\x,2) -- (1.51+\x,2);
};

\foreach \x in {12}
\draw [very thick] 
 (2.4+\x,2.51)  .. controls ++(85:-1) and ++(0:.1) .. (2+\x,1)
  .. controls ++(0:-.1) and ++(-75:1) .. (1.5+\x,2);

\end{tikzpicture}
\caption{This collection of tethered chains yields a 3-simplex in the Hatcher-Vogtman complex,
it determines a Riemann surface with boundary and a canonical isomorphism of its   fundamental group with a free group of rank 8.}
\label{simplex}
\end{figure}
\bigskip

\begin{rem}
We explain now the assertion made in Figure \ref{simplex}, see especially Figure \ref{basis}.

Each handle has a canonical orientation, hence the path joining the point on the red segment with the 
meridian circle is oriented, and we may proceed with a right turn to reach
the point of intersections of the two circles, the meridian and the parallel, which are both naturally oriented.
This determines a canonical free basis of the fundamental group of each handle.

After that, we apply the first van Kampen Theorem to derive the canonical structure of a free product for the union of the several handles.

\end{rem}

\begin{figure}
\begin{tikzpicture}[scale=1.2]


\draw [very thick, dashed] 
 (-1.6,2.5)  .. controls ++(85:-1) and ++(0:.1) .. (-2,1);
\draw [very thick] 
 (-1.6,2.51)  .. controls ++(85:-.2) and ++(82:.2) .. (-1.66,2);
\draw [very thick] 
 (-2,1) .. controls ++(0:-.1) and ++(-75:1) .. (-2.5,2);
 
\draw [very thick] 
 (2.4,2.51)  .. controls ++(85:-1) and ++(0:.1) .. (2,1)
  .. controls ++(0:-.1) and ++(-75:1) .. (1.5,2);
 
\draw [cyan, very thick] 
 (-0.5,2) 
 .. controls ++(105:-1) and ++(45:.2) .. 
 (-0.3,.7)
 .. controls ++(45:-1.3) and ++(0:-1.1) .. 
(0,-2.5)
 .. controls ++(0:1.1) and ++(-38:1.2) .. 
(0.34,.6)
 .. controls ++(-38:-.5) and ++(106:-1) .. 
(-0.3,2);

\draw [purple, very thick] 
 (-0.4,2) 
 .. controls ++(105:-.2) and ++(142:.4) .. 
 (0.12,.65)
 .. controls ++(142:-.65) and ++(102:.2) .. 
 (0.8,-.38)
 .. controls ++(102:-.2) and ++(108:-.1) .. 
 (0.4165,-0.5);

\draw [purple, very thick] 
 (-0.2,2) 
 .. controls ++(105:-.2) and ++(145:.4) .. 
 (0.32,.73)
 .. controls ++(145:-.6) and ++(102:.2) .. 
 (.95,-.35)
 .. controls ++(102:-.2) and ++(100:-.1) .. 
(1.22,-.3);

\draw [purple, dashed, very thick]   
 (0.4165,-0.5) .. controls ++(108:.2) and ++(100:.2) .. (1.22,-.3);

\draw [cyan, very thick]
(-.5,1.8) -- (-.395,1.62);
\draw [cyan, very thick]
(-.41,1.825) -- (-.395,1.62);

\draw [purple, very thick]
(-.215,1.8) -- (-.11,1.62);
\draw [purple, very thick]
(-.115,1.825) -- (-.11,1.62);

\draw [very thick] (-1.6,2.5) -- (2.4,2.5);

\draw [very thick] 
 (-2,1)  .. controls ++(-5:2) and ++(0:-2.7) .. (0,-2.9)
  .. controls ++(0:2.7) and ++(-5:-2) .. (2,1);

\draw [very thick] 
 (-0.1,0.1) .. controls ++(0:.2) and ++(90:.7) .. (.5,-1)
 .. controls ++(90:-.7) and ++(0:.2) .. (-0.1,-2.1);

\draw [very thick] 
 (0,0.07) .. controls ++(0:-.2) and ++(90:.7) .. (-.5,-1)
 .. controls ++(90:-.7) and ++(0:-.2) .. (0,-2.07);
 


\draw [red, very thick] (-2.51,2) -- (1.51,2);

\end{tikzpicture}
\caption{
If $a$ is the path on the left, and $b$ that on the right,
both with the orientation given by the arrow, then
they form a free basis of the fundamental group with respect
to the base segment
and the boundary oriented with the surface on the left is 
homotopic to the commutator
\[
a b a^{-1} b^{-1}
\]
}
\label{basis}
\end{figure}

We refer to  \cite{h-v} for the main properties and details about these notions, 
but let us recall here that $ \tch_{g', 1}$ is $\frac{1}{2}( {g'}-3 )$-connected (Proposition 5.5, \textit{loc. cit.}).
Note also that $ \Gamma_{g', 1}$ acts simplicially on $ \tch_{g', 1}$.

\begin{defin}\label{G-marked tch}
The complex of $G$-marked tethered chains is the simplicial complex defined as the Cartesian product
\[
X_{g'} :=  \tch_{g', 1} \times G^{2{g'}} \, .
\]
It will be considered with the diagonal left action of $ \Gamma_{g', 1}$. 
\end{defin}

\begin{notation}\label{tch_geobasis}\normalfont
Let us first remark that, for a tethered chain $((\hat{\alpha}, \hat{\beta}), t)$, the orientation on $\hat{\beta}$ is  determined by $t$, 
because $t$ is allowed to intersect $\hat{\beta}$ only from its positive side.
Moreover, the geometric basis $\{ \alpha_1, \beta_1, \ldots , \alpha_{g'} , \beta_{g'}\}$ 
of $\pi_1( \Sigma_{g', 1}, y_0)$ determines ${g'}$ vertices of $ \tch_{g', 1}$, as we explain now. 
For every $i=1, \ldots , {g'}$, we choose a representative of $\alpha_i$ (resp. $\beta_i$) consisting of an arc
$t_i$ joining $y_0$ with the intersection point of a chain $(\hat{\alpha}_i, \hat{\beta}_i)$, then it goes through 
$\hat{\alpha}_i$ (resp. $\hat{\beta}_i$), and comes back to $y_0$ along $t_i$.  The isotopy class of 
$((\hat{\alpha}_i, \hat{\beta}_i), t_i)$
is a vertex of $ \tch_{g', 1}$ and will be denoted by $(\alpha_i, \beta_i)$. 
Finally, for every $p =0, \ldots , {g'}-1$, we denote by $\sigma_p$ the $p$-simplex spanned by 
$(\alpha_1, \beta_1), \ldots , (\alpha_{p+1} , \beta_{p+1})$.

We will denote by ${\rm St}(\sigma_p)$ the stabilizer of $\sigma_p$ in $ \Gamma_{g', 1}$, which, as observed in \cite{h-v},
is isomorphic to $ \Gamma_{{g'}-p-1, 1}$. Moreover, since 
${\rm St}(\sigma_p) \subset {\rm St}(\sigma_{p-1})$, by means of the previous isomorphism 
we have the following inclusion:
\[
 \Gamma_{{g'}-p-1, 1} \subset \Gamma_{{g'}-p, 1} \, .
\]
\end{notation}

\subsection{The spectral sequences}\label{the spectral sequences}
In this subsection we introduce two spectral sequences, $\tilde{E}_{p,q}^r ({g'})$ and $E^r_{p,q} ({g'})$,
 here the superscript $r$ refers to the page of the spectral sequence,
both converging to  the equivariant homology groups $H_{p+q}^{ \Gamma_{g', 1}}(G^{2{g'}})$, if $p+q \leq \frac{1}{2}({g'}-3)$, 
and we compute the differential $d^1$ of the second one, $E^1_{p, q}({g'})$. 
When there is no danger of confusion, to have a less heavy  notation, we will denote $\tilde{E}_{p,q}^r ({g'})$ and $E^r_{p,q} ({g'})$
by $\tilde{E}_{p,q}^r $ and $E^r_{p,q} $, respectively.

We refer to \cite{Brown} for the basics on  equivariant homology and on homology of groups.

For fixed ${g'}$, let  $F_\bullet =(F_p)_{p\geq 0}$ be a free resolution of $\ZZ$  in the category of right 
$\ZZ \Gamma_{g', 1}$-modules,
where $\ZZ \Gamma_{g', 1}$ is the group ring of $ \Gamma_{g', 1}$ and $\ZZ$ is the trivial 
$\ZZ \Gamma_{g', 1}$-module.
Let  $S_\bullet =(S_q:= S_q(X_{g'}))_{q\geq 0}$ be the chain complex of $X_{g'}$ (defined in Def. \ref{G-marked tch}) with integer coefficients.
Then we consider the double complex $C_{\bullet , \bullet}=(C_{p,q})_{p, q \geq 0}$, where
\[
C_{p,q} := F_p \otimes_{\ZZ \Gamma_{{g'}, 1}} S_q \, , \quad p, q \geq 0 \, ,
\]
with the horizontal differential $\partial '$ of degree $(-1, 0)$ (induced by the differential of $F_\bullet$) 
and the vertical one $\partial ''$ of degree $(0, -1)$ (induced by the differential of $S_\bullet$).
The total complex associated to $C_{\bullet , \bullet}$ is defined as $T_\bullet =(T_m)_{m\geq 0}$, where
\[
T_m := \oplus_{p+q=m} C_{p, q} \, , \quad m \geq 0 \, ,
\]
with differential $\partial$ given by $\partial_{| C_{p,q}}:= \partial ' + (-1)^p \partial ''$.

We  filter $T_\bullet$ in two ways, by setting 
\[
\tilde{\Phi}_p (T_m) : = \oplus_{i\leq p} C_{i, m-i} \qquad \mbox{and}
\qquad \Phi_p(T_m) := \oplus_{j\leq p} C_{m-j, j} \, .
\]
These filtrations yield  spectral sequences $\{ \tilde{E}^r({g'}) \}$ and $\{ E^r({g'}) \}$, respectively,
both converging to $H(T_\bullet)$.
For the first one we have that 
\[
\tilde{E}^0_{p, q}({g'}) = \tilde{\Phi}_p (T_{p+q})/\tilde{\Phi}_{p-1} (T_{p+q}) = C_{p, q}
\]
with differential $\tilde{d}^0=\pm \partial ''$, hence 
\[
\tilde{E}^1_{p, q} ({g'}) = H_q(C_{p, \bullet}) = 
\begin{cases}
0 & \mbox{for} \quad 0<q\leq \frac{1}{2}({g'}-3) \, , \\
F_p \otimes_{\ZZ  \Gamma_{g', 1}} \ZZ \langle G^{2{g'}} \rangle & \mbox{for} \quad q=0 \, ,
\end{cases}
\]
by \cite[Prop. 5.5]{h-v}. The differential $\tilde{d}^1 \colon \tilde{E}^1_{p,q} ({g'}) \to \tilde{E}^1_{p-1,q}({g'})$
is induced by $\partial ' \colon C_{p, \bullet} \to C_{p-1, \bullet}$, so
\[
\tilde{E}^2_{p, q} ({g'}) = 
\begin{cases}
0 & \mbox{for} \quad 0<q\leq \frac{1}{2}({g'}-3) \, , \\
H_p^{ \Gamma_{g', 1}} ( G^{2{g'}} )  & \mbox{for} \quad q=0 \, .
\end{cases}
\]
We conclude that
\[
\tilde{E}^2_{p, q} ({g'}) = \tilde{E}^\infty_{p, q} ({g'}) \qquad \mbox{and} \qquad H_{p+q}(T_\bullet)=
H_{p+q}^{ \Gamma_{g', 1}} (G^{2{g'}}) \, , 
\qquad \mbox{for} \quad p+q \leq \frac{1}{2}({g'}-3) \, .
\]

For the second spectral sequence we have that
\[
{E}^0_{p, q} ({g'}) = {\Phi}_p (T_{p+q})/{\Phi}_{p-1} (T_{p+q}) = C_{q, p}
\]
and
\[
E^1_{p, q} ({g'}) = H_q(C_{\bullet, p}) = H_q( \Gamma_{g', 1} , S_p(X_{g'})) \, ,
\]
where the last group is the homology of $ \Gamma_{g', 1}$ with coefficients in the module $S_p(X_{g'})$.
The differential $d^1 \colon E^1_{p, q} ({g'}) \to E^1_{p-1, q}({g'})$ is induced  (up to sign) by 
$\partial '' \colon C_{\bullet, p} \to C_{\bullet, p-1}$.


In the remaining part of the section we  give an explicit expression for $d^1$ (see Proposition \ref{d^1} below).
To this aim we need some preliminary observations.

\begin{lemma}
For every $p\geq 0$, $ \Gamma_{g', 1}$ acts transitively on the set of $p$-simplices of $ \tch_{g', 1}$.
\end{lemma}
\begin{proof}
Let $\sigma$ and $\sigma'$ be two $p$-simplices, and let  
$((\hat{\alpha}_i, \hat{\beta}_i), t_i)$ (respectively $((\hat{\alpha}_i', \hat{\beta}_i'), t_i')$), $i=1, \ldots , p+1$,
classes that span $\sigma$ (respectively $\sigma'$).
As argued in Notation \ref{tch_geobasis}, we can associate to every tethered chain $((\hat{\alpha}_i, \hat{\beta}_i), t_i)$
two elements $\alpha_i , \beta_i \in \pi_1( \Sigma_{g', 1}, y_0)$ (respectively $\alpha_i' , \beta_i'$).
By definition of $ \tch_{g', 1}$, both sets $$\alpha_1 , \beta_1, \ldots , \alpha_{p+1}, \beta_{p+1}, {\rm and } \ 
\alpha_1' , \beta_1', \ldots , \alpha_{p+1}', \beta_{p+1}'$$ are parts of two geometric bases of $\pi_1( \Sigma_{g', 1}, y_0)$
and hence the Dehn-Nielsen-Baer theorem implies that there is a mapping class that sends 
$(\alpha_i , \beta_i)$ to $(\alpha_i' , \beta_i')$, for every $i=1, \ldots , p+1$.
This mapping class clearly maps  the class of $((\hat{\alpha}_i, \hat{\beta}_i), t_i)$ to the one of
$((\hat{\alpha}_i', \hat{\beta}_i'), t_i')$, for every $i$.

\end{proof}

This Lemma implies that,   as a $\ZZ  \Gamma_{g', 1}$-module
$$
S_p \left( X_{g'} \right) = S_p \left(  \tch_{g', 1} \times G^{2{g'}} \right) \cong {\rm Ind}^{\Gamma_{g', 1}}_{{\rm St}(\sigma_p)} 
\ZZ \langle G^{2{g'}} \rangle \, 
$$
 (\cite[Ch. III (5.4)]{Brown}), where ${\rm St}(\sigma_p)$
is the stabilizer of $\sigma_p$ (see Notation \ref{tch_geobasis}) in $ \Gamma_{g', 1}$, and therefore 
applying Shapiro's lemma we conclude that 
$$
E^1_{p, q} ({g'}) = H_q( \Gamma_{{g'}, 1}, S_p(X_{g'})) \cong H_q({\rm St}(\sigma_p), \ZZ \langle G^{2{g'}} \rangle) \, . 
$$
Now, since ${\rm St}(\sigma_p)$ doesn't act on the first $2p+2$ factors of $\ZZ \langle G^{2{g'}} \rangle$
and  ${\rm St}(\sigma_p)\cong  \Gamma_{{g'}-p-1, 1}$ \cite{h-v}, we have the following isomorphism:
$$
E^1_{p,q} ({g'}) \cong \ZZ \langle G^{2p+2}\rangle \otimes H_q( \Gamma_{{g'}-p-1, 1}, \ZZ \langle G^{2{g'}-2p-2}\rangle) \, ,
$$
where the tensor product is understood over $\ZZ$ if not otherwise specified.
The differential $d^1$ is then a homomorphism
$$
d^1 \colon \ZZ \langle G^{2p+2}\rangle \otimes H_q( \Gamma_{{g'}-p-1, 1}, \ZZ \langle G^{2{g'}-2p-2}\rangle)
\to \ZZ \langle G^{2p}\rangle \otimes H_q( \Gamma_{{g'}-p, 1}, \ZZ \langle G^{2{g'}-2p}\rangle) \, ,
$$
where $ \Gamma_{{g'}-p-1, 1}$ is considered a subgroup of  $ \Gamma_{{g'}-p, 1}$ as explained in Notation \ref{tch_geobasis}.

Note that, for any $(a, b) \in G^2$, the homomorphism
$\ZZ \langle G^{2{g'}-2p-2}\rangle \to \ZZ \langle G^{2{g'}-2p}\rangle$ defined by 
$$
(a_{p+2}, b_{p+2}, \ldots , a_{g'}, b_{g'}) \mapsto (a, b, a_{p+2}, b_{p+2}, \ldots , a_{g'}, b_{g'})
$$
is equivariant with respect to the inclusion $ \Gamma_{{g'}-p-1, 1} \subset \Gamma_{{g'}-p, 1}$, therefore it induces 
a homomorphism in homology, which we denote as follows:
\begin{eqnarray*}
H_*( \Gamma_{{g'}-p-1, 1}, \ZZ \langle G^{2{g'}-2p-2}\rangle) &\to& 
H_*( \Gamma_{{g'}-p, 1}, \ZZ \langle G^{2{g'}-2p}\rangle)\\
m &\mapsto& (a,b)m \, .
\end{eqnarray*}

\begin{prop}\label{d^1}
1. For $p>0$ and $p+q\leq \frac{1}{2}({g'}-3)$,
the differential $d^1$ of the second spectral sequence $E^1_{p,q}({g'})$ is given (up to  sign) by 
the following expression:
\[
d^1 \colon \ZZ \langle G^{2p+2}\rangle \otimes H_q( \Gamma_{{g'}-p-1, 1}, \ZZ \langle G^{2{g'}-2p-2}\rangle)
\to \ZZ \langle G^{2p}\rangle \otimes H_q( \Gamma_{{g'}-p, 1}, \ZZ \langle G^{2{g'}-2p}\rangle) 
\]
\[
(a_1, b_1, \ldots , a_{p+1}, b_{p+1}) \otimes m  \mapsto  
\sum_{k=1}^{p+1}(-1)^{k-1} (a_1, b_1, \ldots , \widehat{a_k, b_k}, \ldots , a_{p+1}, b_{p+1}) \otimes
(a_k^{d_k}, b_k^{d_k})m \, ,
\]
where $d_k := [a_{k+1}, b_{k+1}] \cdot  \ldots \cdot [a_{p+1}, b_{p+1}] $, and $x^y := y^{-1}xy$, for every  $x, y \in G$.

2. 
For $p=0$ and $q\leq \frac{1}{2}({g'}-3)$, the above expression  yields the edge map 
$$
e_q\colon E^1_{0,q} ({g'}) \to H_q( \Gamma_{g', 1}, \ZZ \langle G^{2{g'}}\rangle).
$$ 
\end{prop}
\begin{proof}
1. Let us recall that $d^1$ is induced by the boundary operator 
$\partial '' \colon F_\bullet \otimes_{\ZZ  \Gamma_{g', 1}} S_p  \to  F_\bullet \otimes_{\ZZ  \Gamma_{g', 1}} S_{p-1}$.
Since $ \Gamma_{g', 1}$ acts transitively on the $p$-simplices of $ \tch_{g', 1}$, as observed before we have that
$$
E^1_{p, q} ({g'}) = H_q({\rm St}(\sigma_p), \ZZ \langle G^{2{g'}} \rangle) \, , 
$$
and moreover  
$$
\partial'' \sigma_p = \sum_{k=1}^{p+1}(-1)^{k-1}\sigma_p^k \, ,
$$
where $\sigma_p^k$ is the $(p-1)$-simplex of $ \tch_{g', 1}$ spanned by 
$(\alpha_1, \beta_1), \ldots , \widehat{(\alpha_k, \beta_k)}, \ldots , (\alpha_{p+1} , \beta_{p+1})$
(here we follow Notation \ref{tch_geobasis}), and $\partial''$ is the boundary operator of $S_\bullet$.

As proved in  \cite[VII, (8.1)]{Brown},  $d^1 \colon H_q({\rm St}(\sigma_p), \ZZ \langle G^{2{g'}} \rangle)
\to H_q({\rm St}(\sigma_{p-1}), \ZZ \langle G^{2{g'}} \rangle)$ is given by the following expression (up to a sign):
$$
d^1 = \pm \sum_{k=1}^{p+1} (-1)^{k-1} v_{p,k} \circ u_{p,k} \, ,
$$
where $u_{p, k} \colon H_*({\rm St}(\sigma_p), \ZZ \langle G^{2{g'}} \rangle) 
\to H_*({\rm St}(\sigma_p^k), \ZZ \langle G^{2{g'}} \rangle)$
is the homomorphism induced by the inclusion ${\rm St}(\sigma_p) \subset {\rm St}(\sigma_p^k)$.
The homomorphism  
$v_{p, k} \colon H_*({\rm St}(\sigma_p^k), \ZZ \langle G^{2{g'}} \rangle) \to H_*({\rm St}(\sigma_{p-1}), \ZZ \langle G^{2{g'}} \rangle)$
is associated to the group isomorphism  
\[
{\rm St}(\sigma_p^k) \to {\rm St}(\sigma_{p-1}) \, , \qquad \gamma \mapsto  \gamma^{\gamma_{p,k}} \, , 
\]
(where $\gamma_{p,k} \in { \Gamma_{g', 1}}$ is any element  such that $\gamma_{p, k}^{-1}\sigma_p^k = \sigma_{p-1}$) 
and to
the action of $\gamma_{p, k}^{-1}$ on $\ZZ \langle G^{2{g'}} \rangle$
(notice that 
$\gamma_{p, k}^{-1} \colon \ZZ \langle G^{2{g'}} \rangle \to \ZZ \langle G^{2{g'}} \rangle$ 
 is equivariant with respect to the isomorphism ${\rm St}(\sigma_p^k) \to {\rm St}(\sigma_{p-1})$, 
$\gamma \mapsto  \gamma^{\gamma_{p,k}}$).
 
Therefore $v_{p,k} \circ u_{p,k}$ is the homomorphism induced by the composition  of the following chain maps:
\begin{equation}\label{d1pk}
F_\bullet  \otimes_{{\rm St}(\sigma_p)} \ZZ \langle G^{2{g'}} \rangle \to  
F_\bullet \otimes_{{\rm St}(\sigma_{p}^k)} \ZZ \langle G^{2{g'}} \rangle 
\to F_\bullet \otimes_{{\rm St}(\sigma_{p-1})} \ZZ \langle G^{2{g'}} \rangle  
\end{equation}
\[
x \otimes (a_1, b_1, \ldots , a_{g'}, b_{g'}) \mapsto  x \otimes (a_1, b_1, \ldots , a_{g'}, b_{g'})
\mapsto  x\gamma_{p, k}^{-1} \otimes  \gamma_{p, k}(a_1, b_1, \ldots , a_{g'}, b_{g'}) \, ,
\]
for the last map see \cite[III.8]{Brown} (paragraph preceding III.8.1). 
Note that we can use the same  $F_\bullet$ because it is a free resolution of $\ZZ$  over $\ZZ{\rm St}(\sigma)$, 
for every $\sigma \in { \tch_{g', 1}}$.

We choose $\gamma_{p, k}^{-1}$ such that it induces the following automorphism of $\pi_1({ \Sigma_{g', 1}}, y_0)$:
\begin{equation}\label{gammapk}
\begin{cases} 
& \alpha_i , \beta_i \mapsto \alpha_i , \beta_i \, , \qquad {\rm for} \quad 1\leq i < k \quad {\rm or} \quad p+1< i \leq {g'} \, , \\
& \alpha_i , \beta_i \mapsto \alpha_{i-1} , \beta_{i-1} \, , \qquad {\rm for} \quad k+1\leq i \leq p+1 \, , \\
& \alpha_k , \beta_k \mapsto \alpha_{p+1}^{\delta_k^{-1}} , \beta_{p+1}^{\delta_k^{-1}} \, ,
\end{cases}
\end{equation}
where $\delta_k= [\alpha_k, \beta_k] \cdot \ldots \cdot [\alpha_p, \beta_p]$. 
With this choice, $\gamma_{p, k}$ acts on the set 
$G^{2{g'}}=\operatorname{Hom}(\pi_1({ \Sigma_{g', 1}}, y_0), G)$
as follows:
\begin{equation}\label{gammapk action}
\gamma_{p,k}(a_1,  \ldots ,  b_{g'}) = (a_1,  \ldots,  b_{k-1}, a_{k+1},  \ldots ,  b_{p+1},
a_{k}^{d_k}, b_k^{d_k}, a_{p+2},  \ldots ,  b_{g'}) \, .
\end{equation}

Furthermore, let us note that ${\rm St}(\sigma_p)$ (respectively ${\rm St}(\sigma_{p-1})$) 
acts trivially on the first $2p+2$ (respectively $2p$) entries of 
$(a_1,  \ldots ,  b_{g'})$, hence 
$$
F_\bullet  \otimes_{{\rm St}(\sigma_p)} \ZZ \langle G^{2{g'}} \rangle \cong \ZZ \langle G^{2p+2} \rangle \otimes
F_\bullet  \otimes_{{\rm St}(\sigma_p)} \ZZ \langle G^{2{g'}-2p-2} \rangle
$$
and 
$$
F_\bullet  \otimes_{{\rm St}(\sigma_{p-1})} \ZZ \langle G^{2{g'}} \rangle \cong \ZZ \langle G^{2p} \rangle \otimes
F_\bullet  \otimes_{{\rm St}(\sigma_p)} \ZZ \langle G^{2{g'}-2p} \rangle \, ,
$$
as  complexes. 

Under these identifications we rewrite \eqref{d1pk} as follows:
\begin{equation}\label{d1pk'}
\ZZ \langle G^{2p+2} \rangle \otimes F_\bullet  \otimes_{{\rm St}(\sigma_p)} \ZZ \langle G^{2{g'}-2p-2} \rangle \to 
\ZZ \langle G^{2p} \rangle \otimes F_\bullet  \otimes_{{\rm St}(\sigma_{p-1})} \ZZ \langle G^{2{g'}-2p} \rangle 
\end{equation}
\[
(a_1, \ldots , b_{2p+2}) \otimes x \otimes (a_{2p+3}, \ldots , b_{g'}) \mapsto 
(a_1, \ldots , \widehat{a_k, b_k}, \ldots ,b_{2p+2})
\otimes x\gamma_{p,k}^{-1} \otimes (a_k^{d_k}, b_{k}^{d_k}, \ldots ,  b_{g'}) \, . 
\]

To conclude the proof note that $\gamma_{p,k}$ is centralized by ${\rm St}(\sigma_p)$, therefore 
the map 
$F_\bullet  \otimes_{{\rm St}(\sigma_p)} \ZZ \langle G^{2{g'}-2p-2} \rangle \to 
F_\bullet  \otimes_{{\rm St}(\sigma_p)} \ZZ \langle G^{2{g'}-2p-2} \rangle$ that sends
$x\otimes_{{\rm St}(\sigma_p)} (a_{2p+3}, \ldots , b_{g'})$ to 
$x\gamma_{p, k}^{-1} \otimes_{{\rm St}(\sigma_p)} (a_{2p+3}, \ldots , b_{g'})$
is an isomorphism of complexes and it induces the identity in homology
(in fact the map $F_\bullet \to F_\bullet$, $x\mapsto x\gamma_{p,k}^{-1}$ 
lifts the identity of $\ZZ$, hence it is homotopic to the identity
\cite[Ch. I (7.4)]{Brown}). 
Hence  \eqref{d1pk'} and the following chain map 
\[
\ZZ \langle G^{2p+2} \rangle \otimes F_\bullet  \otimes_{{\rm St}(\sigma_p)} \ZZ \langle G^{2{g'}-2p-2} \rangle \to 
\ZZ \langle G^{2p} \rangle \otimes F_\bullet  \otimes_{{\rm St}(\sigma_{p-1})} \ZZ \langle G^{2{g'}-2p} \rangle  
\]
\[
(a_1, \ldots , b_{2p+2}) \otimes x \otimes (a_{2p+3}, \ldots , b_{g'}) \mapsto 
(a_1, \ldots , \widehat{a_k, b_k}, \ldots ,b_{2p+2})
\otimes x \otimes (a_k^{d_k}, b_{k}^{d_k}, \ldots ,  b_{g'}) 
\]
induce the same homomorphism in homology. 

2. The filtration $\Phi$ of $T_\bullet$ defined before induces the following filtration of the homology $H(T_\bullet)$,
$\Phi_p\left( H_{p+q}(T_\bullet) \right) = {\rm Im}\left( H_{p+q}(\Phi_pT_\bullet) \to H_{p+q}(T_\bullet) \right)$.
The associated graded group is given as follows:
$$
{\rm Gr}_p H_{p+q}(T_\bullet)= \left({\Phi_p(T_{p+q}) \cap Z}\right)/\left({\Phi_p(T_{p+q}) \cap B + \Phi_{p-1}(T_{p+q}) \cap Z}\right) =
E^\infty_{p, q} ({g'})\, .
$$
In particular, for $p=0$ we have that 
$$
{\rm Gr}_0 H_q(T_\bullet)= \left({\Phi_0(T_q) \cap Z}\right) / \left({\Phi_0(T_q) \cap B }\right) =E^\infty_{0, q} ({g'}) \, .
$$

As observed before we have that  $H_{p+q}(T_\bullet)=H_p({ \Gamma_{g', 1}}, \ZZ \langle G^{2{g'}} \rangle)$, 
for   $p+q\leq \frac{1}{2}({g'}-3)$.
Therefore, for every  $q$ such that $q\leq \frac{1}{2}({g'}-3)$, 
the edge homomorphism of the second spectral sequence, 
$$
e_q \colon E^1_{0,q} ({g'}) = H^{\partial'}_q(F \otimes_{ \Gamma_{g', 1}} C_0) \to 
H_q({ \Gamma_{g', 1}}, \ZZ \langle G^{2{g'}} \rangle) \, ,
$$
associates to any $\partial'$-cycle of $\Phi_0(T_q)=C_{q, 0}$ its class 
modulo $\partial' (C_{q+1,0} ) + \partial'' (C_{q, 1})$. Hence $e_q=d^1$ when $p=0$.
\end{proof}

\section{$\mathcal{K}$-complexes}
In this section we introduce the so-called ``Koszul-like complexes" (or $\mathcal{K}$-complexes),
which are chain complexes that have the same structure of the second spectral sequence defined in Section \ref{ccss}
(as described in Proposition \ref{d^1}). We also study some homological algebraic properties of these complexes.
Here we will follow   \cite[Sec. 4.]{e-v-w}.
$\mathcal{K}$-complexes are associated to graded modules over the so-called ``ring of connected components" $R$,
which we  define now.

\subsection{ The ring of connected components}\label{ring of connected components}
\label{ringofconnected}

Let $ \Sigma_{g_1', 1}$ and $ \Sigma_{g_2', 1}$ be two compact oriented surfaces of genus respectively ${g_1'}$ and ${g_2'}$
and with one boundary component.
Let us fix geometric  basis $\{ \tilde{\alpha}_1, \tilde{\beta}_1, \ldots , \tilde{\alpha}_{g_1'}, \tilde{\beta}_{g_1'} \}$ of 
$\pi_1({ \Sigma_{{g_1', 1}}}, \tilde{y}_0)$
and $\{ \alpha_1, \beta_1, \ldots , \alpha_{g'_2}, \beta_{g_2'} \}$ of $\pi_1({ \Sigma_{{g_2'}, 1}}, y_0)$.
If we glue $ \Sigma_{g_1', 1}$ and $ \Sigma_{g_2', 1}$ together, along two intervals in the boundaries, in such a way that 
$\tilde{y}_0$ coincides with $y_0$ 
{ as shown in Figures \ref{tethered} and \ref{simplex}},
we obtain a surface $ \Sigma_{{g_1'}+{g_2'}, 1}$ together with the geometric basis 
$\{ \tilde{\alpha}_1, \tilde{\beta}_1 , \ldots , \tilde{\alpha}_{g_1'}, \tilde{\beta}_{g_1'}, 
\alpha_1, \beta_1, \ldots , \alpha_{g_2'}, \beta_{g_2'} \}$ of $\pi_1({ \Sigma_{{g_1'}+{g_2'}, 1}}, y_0)$.

The natural homomorphisms $ \Gamma_{g_1', 1}, \Gamma_{g_2', 1} \to \Gamma_{{g_1'}+{g_2'}, 1}$, obtained by extending 
every homeomorphism of $ \Sigma_{g_1', 1}$ (respectively of $ \Sigma_{g_2', 1}$) 
to $ \Sigma_{{g_1'}+{g_2'}, 1}$ by the identity,
are injective by the Dehn-Nielsen-Baer Theorem. Therefore we identify $ \Gamma_{g_1', 1}$
and $ \Gamma_{g_2', 1}$ as subgroups of $ \Gamma_{{g_1'}+{g_2'}, 1}$. 
Since these subgroups commute, we obtain an inclusion
$$
 \Gamma_{g_1', 1} \times \Gamma_{g_2', 1} \hookrightarrow \Gamma_{{g_1'}+{g_2'}, 1} \, .
$$
Note that the map 
\begin{eqnarray*}
G^{2{g_1'}} \times G^{2{g_2'}} &\to& G^{2{g_1'}+2{g_2'}} \\
(\tilde{a}_1, \tilde{b}_1, \ldots , \tilde{a}_{g_1'}, \tilde{b}_{g_1'}) \times (a_1, b_1, \ldots , a_{g_2'}, b_{g_2'}) &\mapsto& 
(\tilde{a}_1, \tilde{b}_1, \ldots , \tilde{a}_{g_1'}, \tilde{b}_{g_1'}, a_1, b_1, \ldots , a_{g_2'}, b_{g_2'})
\end{eqnarray*}
is equivariant with respect to this inclusion, therefore it induces a group-homomorphism
\begin{eqnarray*}
\ZZ \langle G^{2{g_1'}}/{ \Gamma_{g_1', 1}} \rangle \times \ZZ \langle G^{2{g_2'}}/{ \Gamma_{g_2', 1}} \rangle &\to& 
\ZZ \langle G^{2{g_1'}+2{g_2'}}/{ \Gamma_{{g_1'}+{g_2'}, 1}} \rangle   \\
\llbracket \tilde{a}_1, \tilde{b}_1, \ldots , \tilde{a}_{g_1'}, \tilde{b}_{g_1'}\rrbracket  \times 
\llbracket a_1, b_1, \ldots , a_{g_2'}, b_{g_2'}\rrbracket  &\mapsto& 
\llbracket \tilde{a}_1, \tilde{b}_1, \ldots , \tilde{a}_{g_1'}, \tilde{b}_{g_1'}, a_1, b_1, \ldots , a_{g_2'}, b_{g_2'}\rrbracket 
\end{eqnarray*}

Let us denote by $R_{g'}:= \ZZ \langle G^{2{g'}}/{ \Gamma_{g', 1}} \rangle$, for ${g'}>0$, and $R_0:= \ZZ$.
Then the previous homomorphisms yield a non-commutative ring structure on 
$$
R:= \oplus_{{g'} \geq 0} R_{g'} \, .
$$
This graded ring is called the \textit{ring of connected components}, it depends on the group $G$.

\begin{rem}\label{A(R)finite}
We have group homomorphisms 
\begin{eqnarray*}
R_{g'} &\to& R_{{g'}+1} \\
\llbracket  a_1, b_1, \ldots , a_{g'}, b_{g'} \rrbracket  &\mapsto& \llbracket  1, 1, a_1, b_1, \ldots , a_{g'}, b_{g'} \rrbracket  \, ,
\end{eqnarray*}
which are isomorphisms for every ${g'}$ greater than a certain integer \cite{DT}. 
In \cite{CLP16} we have extended this result to the case of ramified coverings.
\end{rem}

\begin{notation}\normalfont
We denote by $R_{>0}$ the two-sided ideal $\oplus_{{g'}>0}R_{g'}$ of $R$. Moreover we endow $\ZZ$ with a structure 
of left and right $R$-module by identifying it  with $R/{R_{>0}}$.

In this article all graded $R$-modules are concentrated in non-negative degrees.
\end{notation}

\subsection{ The $\mathcal K$-complex}
\begin{defin}\label{defKM}
Let $R$ be the ring of connected components associated to the group $G$.
Let $M=\oplus_{{g'}\in \ZZ} M_{g'}$ be a graded left $R$-module (by convention $M_{g'} = \{ 0 \}$ for ${g'}<0$). 
The $\mathcal{K}$-complex of $M$
is the complex $(\mathcal{K}(M), d)$, where the term of degree $p$,  $\mathcal{K}(M)_p$, is $0$, if $p<0$, and
$$
\mathcal{K}(M)_p := \ZZ \langle G^{2p} \rangle \otimes M[p] \, , \qquad {\rm if} \quad p\geq 0 \, ,
$$
where $M[p]$ is $M$ shifted by $p$ (i.e. $(M[p])_{g'} = M_{{g'}-p}$), and by definition 
$\ZZ \langle G^{2p} \rangle$ has degree $0$. The differential $d$ is defined as follows:
\[
d \colon \ZZ \langle G^{2p+2} \rangle \otimes M[p+1] \to \ZZ \langle G^{2p} \rangle \otimes M[p] 
\]
\[
 (a_1, b_1, \ldots , a_{p+1}, b_{p+1}) \otimes m  \mapsto  
\sum_{k=1}^{p+1}(-1)^{k-1} (a_1, b_1, \ldots , \widehat{a_k, b_k}, \ldots , a_{p+1}, b_{p+1}) \otimes
\llbracket a_k^{d_k}, b_k^{d_k} \rrbracket  m \, ,
\]
where (as usual) $\widehat{a_k, b_k}$ means that we omit $a_k, b_k$ from $(a_1, b_1, \ldots , a_{p+1}, b_{p+1})$,
and  $d_k := [a_{k+1}, b_{k+1}]  \cdots [a_{p+1}, b_{p+1}]$.
\end{defin}

\begin{rem}
By standard computations it is easy to verify that $(\mathcal{K}(M), d)$ is a complex.
\end{rem}

In this article our main example  will be the following.
\begin{ex}\label{M(q)}
For every  $q\geq 0$, let 
$$
M(q):= \oplus_{{g'}\geq 0} H_q ({ \Gamma_{g', 1}}, \ZZ \langle G^{2{g'}} \rangle ) \, ,
$$
note that $M(0)_0 = \ZZ$ and $M(q)_0= \{ 0 \}$, if $q>0$, by the Alexander Lemma \cite{FM}.

For every $\ell \geq 0$
and $(\tilde{a}_1, \tilde{b}_1, \ldots , \tilde{a}_\ell, \tilde{b}_\ell) \in G^{2\ell}$, the map
\begin{eqnarray*}
\ZZ \langle G^{2{g'}} \rangle &\to& \ZZ \langle G^{2\ell + 2{g'} } \rangle \\
(a_1, b_1, \ldots , a_{g'}, b_{g'}) &\mapsto& (\tilde{a}_1, \tilde{b}_1, \ldots , \tilde{a}_\ell, \tilde{b}_\ell, a_1, b_1, \ldots , a_{g'}, b_{g'})
\end{eqnarray*}
is equivariant with respect to the inclusion $ \Gamma_{g', 1} \hookrightarrow \Gamma_{\ell+{g'}, 1}$

described at the beginning of subsection \ref{ringofconnected},
hence it induces a group homomorphism 
$H_q ({ \Gamma_{g', 1}}, \ZZ \langle G^{2{g'}} \rangle ) \to H_q ({ \Gamma_{\ell+{g'}, 1}}, \ZZ \langle G^{2\ell +2{g'} } \rangle )$,
for any $q\geq 0$.
Note that, if $\gamma \in \Gamma_{\ell, 1}$, $\gamma(\tilde{a}_1, \tilde{b}_1, \ldots , \tilde{a}_\ell, \tilde{b}_\ell)$
induces the same homomorphism between homology groups.
To see this, let us consider a free resolution of $\ZZ$
in the category of right $\ZZ \Gamma_{\ell+g', 1}$-modules, $F_\bullet \to \ZZ$. 
The homomorphism in homology induced by $(\tilde{a}_1, \tilde{b}_1, \ldots , \tilde{a}_\ell, \tilde{b}_\ell)$
corresponds to the chain map 
\begin{eqnarray}\label{RactsonMq1}
F_\bullet \otimes_{\ZZ \Gamma_{g', 1}} \ZZ \langle G^{2g'}\rangle &\to& 
F_\bullet \otimes_{\ZZ \Gamma_{\ell+g', 1}} \ZZ \langle G^{2\ell + 2g'}\rangle \nonumber \\
x\otimes (a_1, b_1, \ldots , a_{g'}, b_{g'}) &\mapsto& 
x\otimes (\tilde{a}_1, \tilde{b}_1, \ldots , \tilde{a}_\ell, \tilde{b}_\ell, a_1, b_1, \ldots , a_{g'}, b_{g'}) \, .
\end{eqnarray}
On the other hand, the homomorphism in homology induced by $\gamma(\tilde{a}_1, \tilde{b}_1, \ldots , \tilde{a}_\ell, \tilde{b}_\ell)$
corresponds to
\begin{eqnarray}\label{RactsonMq2}
x\otimes (a_1, b_1, \ldots , a_{g'}, b_{g'}) &\mapsto& 
x\otimes (\gamma(\tilde{a}_1, \tilde{b}_1, \ldots , \tilde{a}_\ell, \tilde{b}_\ell), a_1, b_1, \ldots , a_{g'}, b_{g'})\\
&=& x\gamma^{-1} \otimes (\tilde{a}_1, \tilde{b}_1, \ldots , \tilde{a}_\ell, \tilde{b}_\ell, a_1, b_1, \ldots , a_{g'}, b_{g'}) \, . \nonumber
\end{eqnarray}
Here we used the inclusion $\Gamma_{\ell, 1} \hookrightarrow \Gamma_{\ell + g', 1}$
similar to the one above.
Note that the map $F_\bullet \to F_\bullet$, $x\mapsto x\gamma^{-1}$ is a morphism  of chain complexes that lifts the identity 
of $\ZZ$, hence it is homotopic to the identity \cite[Ch. I (7.4)]{Brown}. This homotopy induces a homotopy between 
\eqref{RactsonMq1} and \eqref{RactsonMq2}.

We conclude that, the homomorphism in homology induced by 
$(\tilde{a}_1, \tilde{b}_1, \ldots , \tilde{a}_\ell, \tilde{b}_\ell) \in G^{2\ell}$ depends only on 
its class $\llbracket \tilde{a}_1, \tilde{b}_1, \ldots , \tilde{a}_\ell, \tilde{b}_\ell \rrbracket \in R_\ell$,
and it will be denoted as follows:

\begin{eqnarray*}
 H_q ({ \Gamma_{g', 1}}, \ZZ \langle G^{2{g'}} \rangle ) & \to & 
 H_q ({ \Gamma_{\ell+{g'}, 1}}, \ZZ \langle G^{2\ell + 2{g'}} \rangle ) \\
 m &\mapsto& \llbracket \tilde{a}_1, \tilde{b}_1, \ldots , \tilde{a}_\ell, \tilde{b}_\ell \rrbracket  m \, .
\end{eqnarray*}
This defines  a structure of left $R$-module on $M(q)$. Note that $M(0)=R$.

Let $(\mathcal{K}(M(q)), d)$ be the $\mathcal{K}$-complex of $M(q)$.
Note that, for $p>1$, the   term in degree ${g'}$ of $d \colon \mathcal{K}(M(q))_{p} \to \mathcal{K}(M(q))_{p-1}$
coincides with $d^1 \colon E^1_{p-1,q}({g'}) \to E^1_{p-2, q}({g'})$ (the differential described in Proposition \ref{d^1}).
Therefore, for $p>1$, $E^2_{p-1, q}({g'})$ is equal to the  term in degree ${g'}$ of $H_{p}(\mathcal{K}(M(q)))$.
While, for $p=1$, 
the term in degree ${g'}$ of $d\colon \mathcal{K}(M(q))_1 \to M(q)$ 
is the edge map $E^1_{0, q}({g'}) \to E^\infty_{0, q}({g'})$ (Proposition \ref{d^1}).

\end{ex}

\begin{notation}\label{notation degM}\normalfont
Let $M$ be a graded, left $R$-module.
We consider in this section positively graded  $R$-modules,
for instance we let $M=\oplus_{{g'}\geq 0}M_{g'}$.
\begin{itemize}
\item[i)] $U\colon M \to M$ denotes the group homomorphism $U(m)=\llbracket 1,1\rrbracket m$.
\item[ii)] $H_i(M)$ denotes  the graded abelian group ${\rm Tor}_i^R(\ZZ,M)$. In particular 
\[
H_0(M)={M} / {R_{>0}M}  \, .
\]
\item[iii)] The \textit{degree} of $M$ is defined as  ${\rm deg}(M): = \sup \{ {g'} \, | \, M_{g'}\not= 0\} \in \bN \cup \{ \infty \}$.
\item[iv)] We set 
\begin{itemize}
\item[$\bullet$] $M[U] := \ker (U\colon M \to M)$,
\item[$\bullet$] $A(M) := \max \{ \deg M[U] , \deg \left( M/U(M) \right) \}$, 
\item[$\bullet$] $\tilde{A}(M) := \max \{ \deg (U),  \deg (M[U] ) , \deg \left( M/U(M) \right) \}$, and
\item[$\bullet$] $h_i (M) := {\rm deg} \left( H_i(\mathcal{K}(M))\right)$, where $(\mathcal{K}(M), d)$ is the $\mathcal{K}$-complex
of $M$.
\end{itemize}
\item[v)] We define $\bar{R}:= R/UR$, $\bar{M}:= M/UM$.
\end{itemize}
\end{notation}

\begin{rem}\label{11central}
All over the present article, $U$ has degree $1$. However, in several places we leave indicated $\deg (U)$
because the current framework can be modified in order to study other problems  with different $U$'s.

Moreover, since $\llbracket 1,1\rrbracket  \in R$ is central,  $U\colon M \to M$ induces a  homomorphism of chain complexes 
$\mathcal{K}(M) \to \mathcal{K}(M)$,
which is defined in each degree $p$ by ${\rm Id} \otimes U \colon \mathcal{K}(M)_p \to \mathcal{K}(M)_p$.   
\end{rem}

{

}


\begin{figure}[h]
\begin{tikzpicture}[scale=.7]


\foreach \x in {-1,4,8,13}
{
\draw [very thick, dashed] 
 (-1.6+\x,2.5)  .. controls ++(85:-1) and ++(0:.1) .. (-2+\x,1);
\draw [very thick] 
 (-1.6+\x,2.51)  .. controls ++(85:-.2) and ++(82:.2) .. (-1.66+\x,2);
\draw [very thick] 
 (-2+\x,1) .. controls ++(0:-.1) and ++(-75:1) .. (-2.5+\x,2);
 

\draw [very thick] (-1.6+\x,2.5) -- (2.4+\x,2.5);

\draw [very thick] 
 (-2+\x,1)  .. controls ++(-5:2) and ++(0:-2.7) .. (0+\x,-2.9)
  .. controls ++(0:2.7) and ++(-5:-2) .. (2+\x,1);

\draw [very thick] 
 (-0.1+\x,0.1) .. controls ++(0:.2) and ++(90:.7) .. (.5+\x,-1)
 .. controls ++(90:-.7) and ++(0:.2) .. (-0.1+\x,-2.1);

\draw [very thick] 
 (0+\x,0.07) .. controls ++(0:-.2) and ++(90:.7) .. (-.5+\x,-1)
 .. controls ++(90:-.7) and ++(0:-.2) .. (0+\x,-2.07);
 
\draw [blue] (0+\x,-.9) ellipse (.9 and 1.6);

\draw [blue]   
 (0.4165+\x,-0.5) .. controls ++(108:-.2) and ++(100:-.2) .. (1.22+\x,-.3);
\draw [blue, dashed]   
 (0.4165+\x,-0.5) .. controls ++(108:.2) and ++(100:.2) .. (1.22+\x,-.3);

\draw [blue]   
 (-0.4+\x,2) .. controls ++(100:-.2) and ++(120:.2) .. (0+\x,.7);

\draw [red, very thick] (-2.51+\x,2) -- (1.51+\x,2);
};

\foreach \x in {-1,8,13}
\draw [very thick] 
 (2.4+\x,2.51)  .. controls ++(85:-1) and ++(0:.1) .. (2+\x,1)
  .. controls ++(0:-.1) and ++(-75:1) .. (1.5+\x,2);

\begin{scope}[yshift=-6cm]

\foreach \x in {0,4,8,12}
{
\draw [very thick, dashed] 
 (-1.6+\x,2.5)  .. controls ++(85:-1) and ++(0:.1) .. (-2+\x,1);
\draw [very thick] 
 (-1.6+\x,2.51)  .. controls ++(85:-.2) and ++(82:.2) .. (-1.66+\x,2);
\draw [very thick] 
 (-2+\x,1) .. controls ++(0:-.1) and ++(-75:1) .. (-2.5+\x,2);
 

\draw [very thick] (-1.6+\x,2.5) -- (2.4+\x,2.5);

\draw [very thick] 
 (-2+\x,1)  .. controls ++(-5:2) and ++(0:-2.7) .. (0+\x,-2.9)
  .. controls ++(0:2.7) and ++(-5:-2) .. (2+\x,1);

\draw [very thick] 
 (-0.1+\x,0.1) .. controls ++(0:.2) and ++(90:.7) .. (.5+\x,-1)
 .. controls ++(90:-.7) and ++(0:.2) .. (-0.1+\x,-2.1);

\draw [very thick] 
 (0+\x,0.07) .. controls ++(0:-.2) and ++(90:.7) .. (-.5+\x,-1)
 .. controls ++(90:-.7) and ++(0:-.2) .. (0+\x,-2.07);
 
\draw [blue] (0+\x,-.9) ellipse (.9 and 1.6);

\draw [blue]   
 (0.4165+\x,-0.5) .. controls ++(108:-.2) and ++(100:-.2) .. (1.22+\x,-.3);
\draw [blue, dashed]   
 (0.4165+\x,-0.5) .. controls ++(108:.2) and ++(100:.2) .. (1.22+\x,-.3);


\draw [red, very thick] (-2.51+\x,2) -- (1.51+\x,2);
};

\foreach \x in {12}
\draw [very thick] 
 (2.4+\x,2.51)  .. controls ++(85:-1) and ++(0:.1) .. (2+\x,1)
  .. controls ++(0:-.1) and ++(-75:1) .. (1.5+\x,2);

\foreach \x in {0,4}
\draw [blue]   
 (-0.4+\x,2) .. controls ++(100:-.5) and ++(120:.5) .. (4+\x,.7);

\foreach \x in {12}
\draw [blue]   
 (-0.4+\x,2) .. controls ++(100:-.2) and ++(120:.2) .. (0+\x,.7);

\draw [blue]   
 (-0.4+8,2) 
 .. controls ++(100:-.2) and ++(-120:-.2) 
 .. (9.4,.8);
\draw [blue, dashed]   
(9.4,.8)
 .. controls ++(-25:-3.5) and ++(0:.8) 
 .. (2.05,1);
\draw [blue]   
(2.05,1)
 .. controls ++(0:-.3) and ++(120:.7) ..
(0,.7);

\end{scope} 

\end{tikzpicture}
\caption{ }\label{central}
\end{figure}

\begin{rem}
Figure \ref{central} serves to illustrate the fact that $U$ is central in $R$.

The upper row shows the product of $3$ factors together with
$4$ standard tethered chains.

In the bottom row the glueing has been performed but a new
collection of $4$ tethered chains has been chosen. It is in the
mapping class group orbit of the standard one, since the mapping
class group acts transitively on all simplices of tethered chains.

If the monodromy with respect to the standard system is
\[
(a_1,b_1,a_2,b_2,a_3,b_3,a_4,b_4)
\]
then the monodromy with respect to the new system is
\[
(a_1,b_1,a^t_4,b^t_4,a_2,b_2,a_3,b_3)
\]
where ${}^t$ indicates conjugation by $t=([a_2,b_2][a_3,b_3])^{-1}$

In case the monodromy of the last factor is trivial we get
\[
(a_1,b_1,a_4,b_4,a_2,b_2,a_3,b_3)
\]
We conclude that:
\begin{quote}
Stabilization commutes with product, since
doing the product first and then the stabilization, ie. a product with
a handle of trivial monodromy,
is up to the action of the mapping class group the same
as doing stabilization first and then the product.
\end{quote}
\end{rem}


In order to motivate the remaining part of the article, let us give a rough idea of the proof of our main result.
Here we follow  Notation \ref{notation degM}.
Let $M=M(q)$ be the graded left $R$-module of Example \ref{M(q)},
and let us consider the homomorphism $U\colon M \to M$. 
By $\ZZ [U]$ we denote the polynomial ring 
in $U$ with integer coefficients. 
Notice that $M$ can be considered as a left $\ZZ [U]$-module, where $Um := U(m)$.
Then, homological stability for the homology of degree $q$ holds true if
\[
\deg \left( M/UM \right) \, , \quad \deg \left( M[U] \right) < \infty \, . 
\]
On the other hand
\[
M/UM = \tor_0^{\ZZ[U]}(\ZZ, M)   \quad \mbox{and} \quad M[U]= \tor_1^{\ZZ[U]}(\ZZ, M) \, ,
\]
the strategy is then to relate the homological algebraic properties of $M$ over $R$ with those over $\ZZ[U]$.
This will be achieved in several steps, also considering the  ring $\bar{R}=R/UR$.

\subsection{Homological algebra of $R$}\label{homological algebra of R}

We collect here the main properties of homological algebra that we need about  graded $R$-modules.
We refer to \cite{Brown} and \cite{Weibel} for background material. 

  In particular, let us recall that the category of graded right (resp. left) $R$-modules 
is an abelian category and that every graded right (resp. left) $R$-module $N$ has a free resolution in this category,
$$
\ldots \to F_1 \to F_0 \to N \to 0 \, ,
$$
which we also denote $F_\bullet \to N \to 0$.
Moreover, for a  graded left $R$-module $M$, the  tensor product $\_ \otimes_R M$ is a right exact functor from the category 
of graded right $R$-modules to the category of bi-graded abelian groups and the groups ${\rm Tor}_i^R\left(\_ \ , M \right)$
(for $i\geq 0$) are its left derived functors. Concretely, for any graded right $R$-module $N$, choose 
a free (or projective) graded resolution $F_\bullet \to N \to 0$ as above and define as usual 
\[
{\rm Tor}_i^R\left(N , M \right) : = 
\left( {\ker (F_i \otimes_R M \to F_{i-1} \otimes_R M)}\right) / \left({{\rm im} (F_{i+1} \otimes_R M \to F_{i} \otimes_R M)} \right) \, .
\]
Similarly, $N\otimes_R \_$ is a right exact functor from the category 
of graded left $R$-modules to the category of bi-graded abelian groups, and  its left derived functors are isomorphic to
${\rm Tor}_i^R\left(N , \_ \right)$ (for $i\geq 0$).

It follows from this that, for a ring $S$, if $N$ is also a left $S$-module (resp. $M$ is also a right $S$-module), 
then  ${\rm Tor}_i^R\left(N , M \right)$ is a left $S$-module (resp. a right $S$-module).
This property will be used in the proof of Proposition \ref{EVW4.10}.

Note that ${\rm Tor}_i^R\left( N , M \right)$ has a natural structure of bi-graded abelian group,
${\rm Tor}_i^R \left( N , M \right) = \bigoplus_{k,l \geq 0} {\rm Tor}_i^R \left( N , M \right)_{k,l}$.
Let us set 
$$
{\rm Tor}_i^R \left( N , M \right)_{g'} := \bigoplus_{k+l={g'}} {\rm Tor}_i^R \left( N , M \right)_{k,l} \, ,
$$
then we define
$$
\deg \left({\rm Tor}_i^R \left(N , M \right) \right) := \sup \{ {g'} \, | \, {\rm Tor}_i^R \left( N , M \right)_{g'} \not= 0 \} \, .
$$

\begin{lemma}\label{C1}
Let $M$ be a graded left $R$-module. Then the following assertions are equivalent:
\begin{enumerate}
\item
$\deg (H_0(M)) \leq a$.
\item
$M$ is generated as $R$-module by elements of degree $\leq a$, i.e.
every $m\in M$ can be written as $m=\sum_{i=1}^k r_i m_i$, for some $r_i\in R$ and $m_i\in M$
with $\deg (m_i) \leq a$, $i=1, \ldots , k$.
\end{enumerate}
\end{lemma}
\begin{proof}
This is a direct consequence of the fact that $H_0(M)=M/R_{>0}M$.
\end{proof}

\begin{lemma}\label{4.4/Z}
Let $M$ be a graded left $R$-module and let $N$ be a graded right $R$-module. Then the following inequality holds true:
\begin{equation*}
\deg (N\otimes_R M) \leq \min \{ \deg (N) + \deg (H_0(M)) \, ,  \deg (H_0(N)) + \deg (M) \} \, .
\end{equation*}
\end{lemma}
\begin{proof}
We prove that $\deg (N\otimes_R M) \leq  \deg (N) + \deg (H_0(M))$. A similar argument shows that 
$\deg (N\otimes_R M) \leq   \deg (H_0(N)) + \deg (M)$.

Let  $a:= \deg (H_0(M))$ and suppose that $a< +\infty$ (otherwise the  claim is obvious).
By Lemma \ref{C1},  $M$ is generated by elements of $\oplus_{i=0}^a M_i$.
Therefore every element $n\otimes m \in N\otimes_R M$ has an expression of the  form
\[
n\otimes m = \sum_{i=1}^k nr_i \otimes m_i \, ,
\]
where $r_i\in R$ and $\deg (m_i) \leq a$ for every $i=1, \ldots , k$. The claim follows from the fact that 
$\deg ( nr_i\otimes m_i) \leq \deg (N) + a$, for every $i$. 
\end{proof}

\begin{lemma}\label{res_deg}
Let $\bar N$ be a graded right $\bar R$-module. Then $\bar N$  has a resolution, 
$$
\ldots \to \bar{P}_1 \to \bar{P}_0\to \bar N \to 0 \, ,
$$
where the morphisms $\bar{P}_{i+1} \to \bar{P}_i$ (for every $i\geq 0$) and $\bar{P}_0\to \bar N$
are of degree $0$, and 
 with  $\bar{P}_i$ a graded  free right $\bar R$-module generated by elements of degrees 
less than or equal to $\deg (\bar N) + i \deg (\bar R)$, for all $i\geq 0$.
\end{lemma}
\begin{proof}
We construct such a resolution proceeding inductively on $i$.
For $i=0$, let $\bar{P}_0$ be the free right $\bar R$-module on the elements of $\bar N$
and let $\bar{P}_0\to \bar N$ be the morphism induced by the identity on these generators.
For $i>0$, let $\bar{P}_i$ be the free right $\bar R$-module on the elements of 
$\ker (\bar{P}_{i-1} \to \bar{P}_{i-2})$ (where we set $\bar{P}_{-1}:= \bar N$)
and let $\bar{P}_i \to \bar{P}_{i-1}$ be the morphism induced by the identity on these generators.
By construction,  the morphisms $\bar{P}_{i} \to \bar{P}_{i-1}$,  for every $i\geq 0$, 
are of degree $0$ and $\bar{P}_{i-1}$ is generated by elements of degree less than or equal to
$\deg (\bar N) + (i-1) \deg (\bar R)$. Therefore the elements of $\ker (\bar{P}_{i-1} \to \bar{P}_{i-2})$
have degree less than or equal to $\deg (\bar N) + i \deg (\bar R)$, this completes the proof since these
elements generate $\bar{P}_{i}$.
\end{proof}

\begin{lemma}\label{EVW4.5}
Let $\bar{M}$ be a graded left  $\bar{R}$-module, and let $\bar N$ be a graded right $\bar R$-module.  
Then 
\[
\deg \left( {\rm Tor}_i^{\bar{R}}(\bar N, \bar{M}) \right) \leq   \deg ( \bar N ) + \deg (\bar{M}) + i\deg(\bar{R}) \, , \quad  
\forall i\geq 0 \, .
\] 
\end{lemma}
\begin{proof} 
Let $\ldots \to \bar{P}_1 \to \bar{P}_0\to  \bar N \to 0$
be a resolution as in Lemma \ref{res_deg}. Since ${\rm Tor}_i^{\bar{R}}(\bar N, \bar{M})$ is a sub-quotient of 
$\bar{P}_i \otimes_{\bar{R}} \bar{M}$, we have that 
\[
\deg  \left( {\rm Tor}_i^{\bar{R}}(\bar N, \bar{M}) \right) \leq \deg \left( \bar{P}_i \otimes_{\bar{R}} \bar{M} \right) \, .
\] 
By Lemma \ref{res_deg} we can write every element $x\in \bar{P}_i$ as $x=\sum_{\ell=1}^k \bar{r}_\ell x_\ell$,
where $\bar{r}_\ell \in \bar R$ and $x_\ell \in \bar{P}_i$ have degree $\leq \deg (\bar N) + i \deg (\bar R)$,
$\ell=1, \ldots , k$. So, the elements of $\bar{P}_i \otimes_{\bar{R}} \bar{M}$ of the form $x\otimes m$
can be written as 
$$
x\otimes m = \sum_{\ell=1}^k x_\ell \otimes \bar{r}_\ell m \, , 
$$
and hence they have degree $\leq \deg (\bar N) + i \deg (\bar R) + \deg (\bar M)$. 
\end{proof}

In the following lemmas, for a graded left $R$-module $M$,  we use the following notation: 
\begin{equation}\label{deltaM}
\delta (M) := \max \left\{ \deg \left( \tor^R_0(\bar R, M) \right) \, , \, \deg \left( \tor^R_1 (\bar R, M) \right) \right\} \, .
\end{equation}

\begin{lemma}\label{lemma 3.12}
Let $M$ be a graded left $R$-module. Then 
\[
A(M) \leq  \delta (M) + A(R)  \, .
\]
\end{lemma}

\begin{proof}
Let us consider the following exact sequence of right $R$-modules, 
\[
R\stackrel U\to R\to \bar R\to 0 \, .
\]
Applying the tensor product $\otimes_R M$ to  this sequence yields the isomorphism
\[
\tor_0^R(\bar R, M) = \bar R \otimes_R M = M / U(M) \, .
\]
Therefore $\deg \left( M / UM \right) = \deg (\tor_0^R(\bar R, M)) \leq \delta (M) + A(R)$.

To prove that  $\deg \left( M[U] \right) \leq \delta (M) + A(R)$,
let us first identify $M$ with $R\otimes_R M$ and let us factor the multiplication map $U \colon M \to M$
as follows:
\begin{equation}\label{U_factors}
\begin{array}{cccccccccc}
R\otimes_R M & \stackrel\alpha\to & U(R) \otimes_R M & \stackrel\beta\to & R\otimes_R M \\
s\otimes m & \mapsto & U(s) \otimes m
& \mapsto & U(s)\otimes m \, ,
\end{array}
\end{equation}
where $U(R)$ is considered as a two-sided ideal of $R$.
Note that $\deg (\beta ) =0$ and that $\deg (\alpha ) = \deg (U)=1$. 
Next, let us consider the following short-exact sequences of left $R$-modules:
\[
0\to U(R) \to R \to \bar R \to 0,\qquad
0\to R[U] \to R \to U(R) \to 0 \, \, .
\]
Applying the tensor product $\otimes_R M$ to  these  sequences we obtain two long exact sequences 
of the corresponding $\tor_*^R(\_ , M)$ groups, part of them are as follows:
\begin{equation}
\label{tensor1}
\begin{array}{ccccccccccc}
0 & \to & \tor_1^R(\bar R,M) & \to &
U(R) \otimes_R M & \stackrel\beta\to & R\otimes_R M 
& \to & \bar R \otimes_R M & \to & 0 
\end{array}
\end{equation}

\begin{equation}
\label{tensor2}
\begin{array}{ccccccc}
R[U] \otimes_R M & \to & R\otimes_R M 
& \stackrel\alpha\to & U (R) \otimes_R M & \to & 0 \, .
\end{array}
\end{equation}
By \eqref{tensor2} $R[U] \otimes_R M$ surjects onto $\ker ( \alpha)$, which is isomorphic to a submodule
of $M[U]$. Since $M[U]$ is the kernel of $U$, its image under $\alpha$ is contained in $\ker (\beta)$, 
which by \eqref{tensor1} is identified with $\tor_1^R(\bar R,M)$.
Thus we get a sequence that is exact in the middle:
\[
\begin{array}{ccccccc}
R[U] \otimes_R M & \to & M[U] & \stackrel\alpha\to & \tor_1^R(\bar R,M) \, .
\end{array}
\]
By Lemma \ref{4.4/Z} the degree bound on the left hand side is 
\[
\deg \left( R[U] \otimes_R M \right)  \leq  \deg (R[U] )  + \deg (H_0(M))  = 
\deg (R[U] )  + \deg (M/R_{>0}M) \, .
\]
On the other hand, for an homogeneous element $x\in M[U]$, if $\alpha(x)=0$, then
$\deg (x) \leq \deg \left( R[U] \otimes_R M \right)$, while if $\alpha(x)\not= 0$, then 
$\deg (x) \leq \deg \left( \tor_1^R(\bar R,M) \right) -\deg (U)$. 
Hence
\[
\deg (M[U]) \quad \leq \quad \max\{
\deg (R[U])  + \deg (M/R_{>0}M),\, \deg \left( \tor_1^R(\bar R,M)\right) -\deg (U) \} \, .
\]
Now note that $M/U(M)$  surjects onto $M/R_{>0}M$, and hence
\[
\deg (M/R_{>0}M) \leq \deg (M/U(M)) =
\deg \left( \tor_0^R(\bar R, M) \right) \, .
\]
From this we conclude  that 
\[
\deg (M[U]) \leq \delta (M) + A(R) \, .
\]
\end{proof}

\begin{lemma}\label{EVW4.7}
Let $M$ be a graded left $R$-module.  Then the following inequalities hold true
\[
\deg \left(\tor_i^R(\bar R,M) \right) \leq
\begin{cases}
\delta (M) & \quad \mbox{for} \quad i=0,1 \, ; \\
\delta (M) + i {\tilde{A}(R)} & \quad \mbox{for} \quad i>1 \, . 
\end{cases}
\]

\end{lemma}
\begin{proof}
The first claim is true by \eqref{deltaM} (the  definition of $\delta (M)$),
the second one will be proved by induction.
To this aim, we consider a resolution $P_\bullet$ of $\bar R$ by graded right $R$-modules
(which is not a free or projective resolution), with
\[
P_0 = R, \quad P_1= R[\deg U], \quad
P_1 \stackrel{U}{\longrightarrow} P_0 \longrightarrow \bar R
\longrightarrow 0
\]
(note the degree shift for $P_1$ which is needed to make the multiplication map
by $U$  of degree $0$).
The kernel of the map $P_1 \to P_0$ is then $R[U][\deg U]$, which is indeed a right $\bar R$-module
of degree $\deg U +\deg R[U]$, therefore by Lemma \ref{res_deg}  it has a resolution,
\[
 \ldots {\longrightarrow} P_3 \longrightarrow P_2
\longrightarrow R[U][\deg U] \, ,
\]
by free right $\bar R$-modules $P_i$, which are
generated in degrees up to
\begin{equation}
\label{shift_bound}
\deg U +\deg R[U]+(i-2)\deg \bar R \, .
\end{equation}
Note that any such $P_i$, for $i\geq 2$, is  a direct sum of copies of $\bar R$ shifted by at most 
$\deg U +\deg R[U]+(i-2)\deg \bar R$, hence its degree is at most
\[
\deg U +\deg R[U]+(i-1)\deg \bar R \, .
\]
Moreover $P_i$ is a right graded $R$-module via the quotient homomorphism $R\to\bar R$, so
$P_\bullet \to \bar R$ is a resolution by 
graded right $R$-modules generated in
degree bounded by the formulas above.

Let now $F_\bullet\to M$ be a free resolution by left $R$-modules.
We get a double complex $P_\bullet \otimes_R F_\bullet$, from which
two spectral sequences can be obtained. One degenerates at the second page:
\[
\begin{array}{ccccccccccc}
E^1_{ij} & = & H_i(P_\bullet\otimes F_j)  =  H_i(P_\bullet)\otimes F_j
 &\cong&  \left\{ \begin{array}{cr} \bar R \otimes F_j   , \,  i=0 \\ 0   , \,   i>0 \end{array} \right.
\\
E^2_{ij} & = & \left\{ \begin{array}{cr} H_i(\bar R \otimes F_j )  , \,  i=0 \\ 0  , \,  i>0 \end{array} \right.
 &=&  \left\{ \begin{array}{cr} \tor_i^R (\bar R,M)  , \,  i=0 \\ 0  , \,  i>0 \end{array} \right.
\end{array}
\]
and hence abuts to $\tor_{i+j}^R(\bar R,M)$. For the other one we have, accordingly:
\[
E^1_{ab} \quad = \quad H_a(P_b\otimes F_\bullet) 
\quad = \quad \tor_a^R(P_b, M)
\quad \implies \quad \tor_{a+b}^R(\bar R,M) \, .
\]
Thus for some filtration on $\tor_{i}^R(\bar R,M)$, the associated graded
module is a sub-quotient of 
\[
\bigoplus_{a+b=i} \tor^R_a(P_b,M) \, .
\]
Since  $P_0 = P_1$ are free $R$-modules, $\tor^R_a(P_b,M) = 0$ for $a>0$ and $b=0, 1$.
When $b\geq2$,
\begin{eqnarray*}
\deg \tor^R_a(P_b,M) & \leq &  
\deg U +\deg R[U]+(b-2)\deg \bar R + \deg\tor^R_a(\bar R, M) \, .
\end{eqnarray*}
The last contribution, $\deg\tor^R_a(\bar R, M)$, is due to $P_b$ being 
a sum of copies of $\bar R$ up to shifts, 
while the other contributions are due to the bound \eqref{shift_bound} on that shift.
It follows from this that a bound on $\deg \left( \tor_{i}^R(\bar R,M) \right)$ is given by the maximum over $a=0, \ldots ,i-2$
of the expression on the right hand side:
\[
\max \{  \deg U +\deg R[U]+(i-a-2)\deg \bar R + \deg\tor^R_a(\bar R, M) \, | \, a=0, \ldots ,i-2 \} \, .
\]
By the induction hypothesis, 
$$
\deg \tor^R_a(\bar R, M) \leq \delta (M) + (i-2) \max \{ \deg (U), \deg (R[U]), \deg (\bar R) \} \, , \quad {\rm for} \quad 0\leq a
\leq  i-2 \, .
$$ 
Combining this with the following inequalities
$$
\deg U , \deg R[U] , \deg \bar R \leq \max \{ \deg (U), \deg (R[U]), \deg (\bar R) \}
$$
the claim follows.
\end{proof}

\begin{lemma}\label{EVW 4.8}
Let $\bar M$ be a graded left $\bar R$-module, then 
\[
\deg \bar M \quad \leq \quad \deg( \ZZ \otimes_{\bar R} \bar M)
+ \deg \bar R
\]
\end{lemma}

\begin{proof}
Under the identification $\ZZ = \bar R / {\bar R}_{>0}$,  we have the following exact sequence:
$$
{\bar R}_{>0} \otimes_{\bar R} \bar M \to \bar R \otimes_{\bar R} \bar M \to \ZZ \otimes_{\bar R} \bar M \to 0 \, .
$$
For an element $x\in \bar M$ of $\deg (x) > \deg( \ZZ \otimes_{\bar R} \bar M) + \deg \bar R$,
we have that $1\otimes x = 0$ in  $\ZZ \otimes_{\bar R} \bar M$, therefore $x=a' x'$, for some 
$a' \in {\bar R}_{>0}$ and $x' \in \bar M$ of $\deg (x')< \deg (x)$. Note that
either $\deg (a')> \deg (\bar R)$, or $\deg (x')> \deg (\ZZ \otimes_{\bar R} \bar M)$.
In the first case $x=0$ and in the second one (by the previous argument) $x'=a'' x''$, hence 
$x= a'a''x''$, with $\deg (x'')<\deg (x')$. We have again two cases: 
$\deg (a'a'')> \deg (\bar R)$, or $\deg (x'')> \deg (\ZZ \otimes_{\bar R} \bar M)$.
In the first case $x=0$, in the second one we repeat the previous argument.
The process terminates since at each step we get an element of $\bar M$ of degree less than the previous one,
and finally we have $x=\tilde a \tilde x$, with $\tilde a \in \bar R$ of degree $> \deg (\bar R)$, therefore $x=0$. 
\end{proof}

\begin{lemma}\label{EVW4.9}
Let $M$ be a graded left $R$-module, then
$$
\delta (M) \leq \max \left\{ \deg H_0(M) \, , \deg H_1 (M) \right\} +4A(R) \, .
$$
\end{lemma}
\begin{proof}
Since $\tor^R_0(\bar R , M)$ is isomorphic to $\bar R \otimes_{R}M$, they have the same degree,
and applying Lemma \ref{4.4/Z} we have:
$$
\deg \tor^R_0(\bar R , M) = \deg (\bar R \otimes_{R}M) \leq \deg \bar R + \deg H_0(M) \, .
$$
Therefore it remains to bound $\deg \tor^R_1(\bar R , M)$. To this aim we first apply Lemma \ref{EVW 4.8} and we 
obtain the following inequality:
\begin{equation}\label{star in lemma 3.15}
\deg \tor^R_1(\bar R , M) \leq \deg( \ZZ \otimes_{\bar R} \tor^R_1(\bar R , M)) + \deg \bar R \, .
\end{equation}
To bound $\deg( \ZZ \otimes_{\bar R} \tor^R_1(\bar R , M))$ we apply Grothendieck's spectral sequence 
(\cite[Corollary 5.8.4]{Weibel}) to the following right exact functors:
\begin{eqnarray*}
G\colon \left( \mbox{Graded left $R$-modules} \right) &\to& \left( \mbox{Graded left $\bar R$-modules} \right) \\
M &\mapsto&  \bar R \otimes_R M \, 
\end{eqnarray*}
and
\begin{eqnarray*}
F\colon \left( \mbox{Graded left $\bar R$-modules} \right) &\to& \left( \mbox{Graded $\ZZ$-modules} \right) \\
\bar M &\mapsto& \ZZ \otimes_{\bar R} \bar M   \, .
\end{eqnarray*}
There is a convergent first quadrant homology spectral sequence for each  graded left $R$-module $M$:
$$
E^2_{pq} = (L_pF)(L_qG)(M) \Rightarrow L_{p+q}(F\circ G)(M) \, .
$$
Since $L_qG(M)= \tor_q^R(\bar R , M)$, $L_pF(\bar M) =\tor_p^{\bar R}(\ZZ , \bar M)$  
and $L_{p+q}(F\circ G)(M)=\tor_{p+q}^R(\ZZ , M)$, this spectral sequence reads as follows:
$$
E^2_{pq} =\tor^{\bar R}_p \left( \ZZ , \tor_q^R(\bar R, M) \right) \Rightarrow \tor_{p+q}^R(\ZZ , M) \, .
$$
The exact sequence of low degree terms associated to this spectral sequence is:
\begin{eqnarray}\label{esldt}
&&\tor_2^R(\ZZ , M) \to \tor_2^{\bar R}(\ZZ , \bar R \otimes_R M) \to \ZZ \otimes_{\bar R} \tor_1^R(\bar R, M) \to \nonumber \\ 
&&\to \tor_1^R(\ZZ , M)
\to \tor_1^{\bar R}(\ZZ , \bar R \otimes_R M) \to 0 \, . 
\end{eqnarray}
From this it follows the inequality
\begin{equation*}
 \deg \left( \ZZ \otimes_{\bar R} \tor_1^R(\bar R , M) \right)
 \leq
  \max \left\{ \deg \tor_1^R(\ZZ , M) , \deg \tor_2^{\bar R}(\ZZ , \bar R \otimes_R M) \right\} \, ,
\end{equation*}
which, combined with \eqref{star in lemma 3.15} and using our Notation \ref{notation degM}, ii), yields:
$$
\deg \tor^R_1(\bar R , M) \leq  \deg \bar R +  \max \left\{ \deg H_1(M) , \deg \tor_2^{\bar R}(\ZZ , \bar R \otimes_R M) \right\} \, .
$$
Moreover, 
\begin{eqnarray*}
\deg \tor_2^{\bar R}(\ZZ , \bar R \otimes_R M) &\leq& \deg (\bar R \otimes_R M) +2\deg \bar R  
\hspace{3cm} \text{(Lemma \ref{EVW4.5})} \\
&\leq& \deg \bar R + \deg H_0(M) +2\deg \bar R \hspace{2cm} \text{(Lemma \ref{4.4/Z})}\\
&=& 3\deg \bar R + \deg H_0(M) \, .
\end{eqnarray*}
Combining the previous inequalities we obtain the following ones, that prove the claim:
\begin{eqnarray*}
\delta (M) &\leq& \max \left\{ \deg \bar R + \deg H_0(M) , \deg \bar R+ \max \{ \deg H_1(M), 3\deg \bar R + \deg H_0(M) \} \right\} \\
&=& \deg \bar R + \max \left\{ \deg H_0(M), \max \{ \deg H_1(M), 3\deg \bar R + \deg H_0(M) \} \right\} \\
&\leq& \deg \bar R + \max \left\{ \deg H_1(M), 3\deg \bar R + \deg H_0(M) \right\} \\
&\leq& 4\deg \bar R + \max \left\{ \deg H_1(M),  \deg H_0(M) \right\} \, .
\end{eqnarray*}
\end{proof}

\begin{rem}\label{tor}Let be given a ring homomorphism $R\to S$, a right $S$-module $N$
and a left $R$-module $M$.

Let $P_\bullet\to M$ be a free $R$-resolution of $M$ and
$Q_\bullet\to N$ a free $S$-resolution of $N$.
Consider the double complex $Q_\bullet \otimes_S (S \otimes_R P_\bullet)$:
one spectral sequence has $E^1$-page with entries
\[
H_p(Q_\bullet) \otimes_S (S \otimes_R P_q) \quad = \quad
\begin{cases}
N \otimes_S (S \otimes_R P_q) & \text{if } p=0  \\
0 & \text{if } p>0.
\end{cases}
\]
That spectral sequence degenerates at its $E^2$-page with all entries $0$ except 
when $p=0$, where:
\[
H_q(N \otimes_S (S \otimes_R P_\bullet)) \quad = \quad
H_q(N \otimes_R P_\bullet) \quad = \quad
\tor_q^R(N,M).
\]
The other spectral sequence has $E^1$-page with entries
\[
Q_i \otimes_S H_j(S \otimes_R P_\bullet) \quad = \quad
Q_i \otimes_S \tor_j^R(S,M).
\]
The corresponding $E^2$-page has entries
\[
H_i(Q_\bullet \otimes_S \tor_j^R(S,M)) \quad = \quad
\tor_i^S(N,\tor_j^R(S,M)).
\]
Combining this with the first degenerating one implies
\[
\tor_i^S(N,\tor_j^R(S,M)) \quad \implies \quad
\tor_{i+j}^R(N,M).
\]
\end{rem}

\begin{prop}\label{EVW4.10}
Let $M$ be a graded left $R$-module and let $N$ be a graded right $R$-module.
Then, for every $i\geq 0$,
\[
\deg \left( {\rm Tor}_i^R \left( N , M \right) \right) \leq \deg (N) + \max \{ \deg \left( H_0(M) \right), \deg \left( H_1(M) \right) \}
+ 4A(R) + i \tilde{A}(R)  \, .
\]
\end{prop}
\begin{proof}
Clearly, if $\deg (N)= \infty$, the claim holds true. Hence we assume that $\deg (N) < \infty$ and we proceed by 
induction on $\deg (N)$. 

If $\deg (N) = 0$, then $N \cdot U = 0$ and so $N$ has a structure of graded right $\bar{R}$-module, where $\bar{R} =R/UR$.
By Remark \ref{tor},
which is a full analogue of the  base-change Theorem for Tor (cf. \cite[Thm. 5.6.6.]{Weibel}),  we have  a first quadrant spectral sequence
with $E^2_{p, q} = \tor_p^{\bar{R}} \left( N,\tor_q^R \left(\bar{R},M \right) \right)$ and  
\[
\tor_p^{\bar{R}} \left( N,\tor_q^R(\bar{R},M) \right) \, \Rightarrow \, \tor_{p+q}^R \left( N,M \right) .\, 
\]
Since the degree of a graded  filtered module is the maximum of the degrees
of the graded quotients,  the above spectral sequence argument shows that 

\[
\deg \left( \tor_{i}^R \left( N,M \right) \right) \leq \max \left\{ \deg \left( \tor_p^{\bar{R}} \left( N,\tor_q^R \left( \bar{R},M \right) \right)
\right) \, | \, p+q = i  \right\} \, . 
\]
Applying Lemma \ref{EVW4.5}, we have that the right hand side of the previous inequality is 
\begin{eqnarray*}
&\leq & \max \left\{ \deg (N) + \deg \left( \tor_q^R \left( \bar{R},M \right) \right) + p \deg (\bar{R}) \, | \, p+q=i \right\}  \\
&= &  \max \left\{  \deg \left( \tor_q^R \left( \bar{R},M \right) \right) + p \deg (\bar{R}) \, | \, p+q=i \right\} \, .
\end{eqnarray*}
Using Lemma \ref{EVW4.7} we can bound the previous expression as
\begin{eqnarray*}
&\leq & \max \left\{  \delta (M) + q \tilde{A}(R) + p \deg (\bar{R}) \, | \, p+q=i \right\} \\
&\leq & \delta (M) + i \tilde{A}(R) \, ,
\end{eqnarray*}
where we used that  $\deg \bar R\leq A (R)  \leq \tilde{A}(R) $.
The claim, in the case where $\deg (N) =0$, follows from Lemma \ref{EVW4.9}.

In the general case, 
let us set $a:= \deg (N)$ and assume that the statement holds true for every graded right $R$-module  of degree $<a$.
Consider the short exact sequence 
\begin{equation}\label{ses4.10}
0 \to N_a \to N \to N/N_a \to 0 \, ,
\end{equation}
where $N_a \subseteq N$ is the summand of degree $a$ of $N$, placed in degree $a$. 
By the  inductive assumption
\[
\deg \left( {\rm Tor}_i^R \left( N/N_a , M \right) \right) \leq a-1 + \max \{ \deg \left( H_0(M) \right), \deg \left( H_1(M) \right) \}
+ 4A(R) + i \tilde{A}(R) \, .
\]
Moreover, $N_a \cdot U =0$,   hence  $N_a $  has a structure of graded right $\bar{R}$-module. 
The same proof of the case where $\deg (N) =0$ yields that 
\[
\deg \left( {\rm Tor}_i^R \left( N_a , M \right) \right) \leq a + \max \{ \deg \left( H_0(M) \right), \deg \left( H_1(M) \right) \}
+ 4A(R) + i \tilde{A}(R) \, .
\]
The inductive statement   follows then from the long exact sequence  for $\tor_\bullet ^R \left( \_ , M \right)$
associated to \eqref{ses4.10}, and from the  principle that  the degree of a graded   module
appearing in an exact sequence is smaller or equal to
  the maximum of the degrees
of the graded modules appearing on the right, respectively on the left of it.

\end{proof}
Let us now consider the  $\mathcal{K}$-complex of $\mathcal{K}(R)$.

\begin{lemma}\label{R>0killsHKR}
For every $x\in R_{>0}$, the right multiplication by $x$ on  $R$ gives a homomorphism of chain complexes
$\mathcal{K}(R) \to \mathcal{K}(R)$, which induces the zero-map in homology.  
\end{lemma}
\begin{proof}

The proof is similar but not identical  to the one in \cite[4.11]{e-v-w}.

A $\ZZ$-module basis of $\sK (R) _{p+1}$ is given by elements of the form
\[
(a_1,b_1,\dots, a_{p+1},b_{p+1}) \llbracket g_1,h_1,\dots,g_q,h_q \rrbracket 
\]
For a pair $(g,h)\in G\times G$ define a map $S_{(g,h)}:\sK (R) \to \sK (R) $ by
\[
(a_1,b_1,\dots,a_{p+1},b_{p+1}) \llbracket g_1,h_1,\dots,g_q,h_q \rrbracket 
 \mapsto 
(g^{\tau^{-1}},h^{\tau^{-1}},a_1,b_1,\dots,a_{p+1},b_{p+1}) \llbracket g_1,h_1,\dots,g_q,h_q \rrbracket 
\]
where $\tau$ is the product of all the commutators appearing, namely, 
\[
\tau \quad : = \quad [a_1,b_1]\cdots[a_{p+1},b_{p+1}][g_1,h_1]\cdots[g_q,h_q]
\]
and one can check that the product is the same, independent of the choice of   all representatives
of the equivalence class $ \llbracket g_1,h_1,\dots,g_q,h_q \rrbracket $.

Then it is a standard exercise to compute
\begin{eqnarray*}
& & \big( S_{(g,h)} d + d S_{(g,h)} \big)\bigg(
(a_1,b_1,\dots,a_{p+1},b_{p+1}) \llbracket g_1,h_1,\dots,g_q,h_q \rrbracket \bigg)\\
& = &
(a_1,b_1,\dots,a_{p+1},b_{p+1})
 \llbracket g^{\tau^{-1}[a_1,b_1]\cdots[a_{p+1},b_{p+1}]},h^{\tau^{-1}[a_1,b_1]\cdots[a_{p+1},b_{p+1}]},
g_1,h_1,\dots,g_q,h_q \rrbracket \\
& = &
(a_1,b_1,\dots,a_{p+1},b_{p+1})
 \llbracket g^{([g_1,h_1]\cdots[g_q,h_q])^{-1}
},h^{([g_1,h_1]\cdots[g_q,h_q])^{-1}
},
g_1,h_1,\dots,g_q,h_q \rrbracket \\
& = &
(a_1,b_1,\dots,a_{p+1},b_{p+1}) \llbracket g_1,h_1,\dots,g_q,h_q, g, h \rrbracket \\
& = &
(a_1,b_1,\dots,a_{p+1},b_{p+1}) \llbracket g_1,h_1,\dots,g_q,h_q \rrbracket \cdot
\llbracket g, h \rrbracket
\end{eqnarray*}
The first equality is a straightforward  application of the definitions of $S_{(g,h)}$ and $d$, followed by a cancellation
which leaves only out the first term of the sum corresponding to $dS_{(g,h)}$, 
which is exactly the term appearing  in the second row.\\
The third  equality is due to a sequence of Zimmermann moves (2.2, page 250  in  \cite{Zimmer}), which allows to swap adjacent
pairs of group elements, provided that one of them is conjugated by the commutator of the other.\\

Thus $S_{(g,h)}$ provides a chain homotopy between the zero map and 
the multiplication by $\llbracket g, h \rrbracket$ on the right.
The claim then follows, since $R_{>0}$ is generated by such elements.

\end{proof}

\begin{prop}\label{hpR}
For every $p\geq 0$ we have that 
$$
h_p(R) \leq p + A(R) +1 \, .
$$
\end{prop}
\begin{proof}
 Direct inspection of the formulae in Definition 3.3, keeping in mind that $U$ acts on the ring element $m$, 
show indeed that  $d_q U = U d_q$ (see also Remark \ref{11central}). 

Hence we have  the following commutative diagram
\[
\begin{tikzcd}
\ZZ \langle G^{2p+2} \rangle \otimes R[p+1] \arrow{r}{d} & \ZZ \langle G^{2p} \rangle \otimes R[p] \arrow{r}{d}
&  \ZZ \langle G^{2p-2} \rangle \otimes R[p-1] \\
\ZZ \langle G^{2p+2} \rangle \otimes R[p+1] \arrow{u}{{\rm Id} \otimes U} \arrow{r}{d} & 
\ZZ \langle G^{2p} \rangle \otimes R[p] \arrow{u}{{\rm Id} \otimes U}  \arrow{r}{d}
&  \ZZ \langle G^{2p-2} \rangle \otimes R[p-1] \arrow{u}{{\rm Id} \otimes U} 
\end{tikzcd}
\]
where on each row  we have written the same part of the $\mathcal{K}$-complex of $R$.

Since $U\colon R_{{g'}-p} \to R_{{g'}-p+1}$ is an isomorphism for ${g'}-p> A(R)$, 
also ${\rm Id} \otimes U$ is an isomorphism for ${g'}-p> A(R)$.
By the commutativity of the diagram, the restriction 
${\rm Id} \otimes U \colon \ker (d) \to \ker (d)$ is an isomorphism in the same range. 
Therefore $\ker (d)$, as a right $R$-module, is generated by elements of degree $\leq p+A(R)+1$.
In particular, also $H_p(\mathcal{K}(R))$ is a right $R$-module generated by elements in degree $\leq p+A(R)+1$.
By the previous Lemma \ref{R>0killsHKR}, $H_p(\mathcal{K}(R))\cdot R_{>0} =0$, hence the claim follows. 
\end{proof}

{
\begin{lemma}\label{universalcoefficients}
Let $M$ be a left graded $R$-module. Then
\begin{equation*}
\deg H_p(\sK(M)) \leq 
\max\{\deg H_0(M), \deg H_1(M)\}
+ (5+p)  \tilde{A}(R)
+ \deg U \, , \quad \forall \, p\geq 0.
\end{equation*}
\end{lemma}
\begin{proof}
Setting as usual  $\sK(M): =\sK(R)\otimes_R M$
($\sK(R)$ is in fact a right and left $R$-module)
we consider (as in \cite[4.13]{e-v-w}) the universal coefficients spectral sequence
\[
\tor_i^R ( H_{p-i}(\sK(R)), M ) \implies H_p (\sK(M)).
\]
Here is how that spectral sequence argument works.
Given the chain complex $\sK(R)_\bullet$ of \emph{free} $R$-modules and the
$R$-module $M$, choose a free resolution $P_\bullet\to M$ of $M$ by $R$-modules.
Then the double chain complex $\sK(R)_\bullet \otimes_R P_\bullet$ gives rise to a
spectral sequence with $E^1$-page entries
\[
H_i(\sK(R)_\bullet ) \otimes_R P_j
\]
and $E^2$-page entries
\[
\tor^R_j (H_i(\sK(R)_\bullet ), M) \, .
\]
The double complex gives also rise to a second spectral sequence. The $E^1$-page entries
are given as follows and can be simplified, since each $\sK(R)_p$ is a free $R$-module.
\[
\tor^R_q(\sK(R)_p,M)\quad = \quad
\begin{cases}
\sK(R)_p\otimes_R M = \sK(M)_p & \text{if } q=0 \\
0 & \text{if } q>0
\end{cases}
\]
Therefore the entries of the corresponding $E^2$-page are
\[
E^2_{q,p} \quad = \quad 
\begin{cases}
H_p(\sK(M)_\bullet) & \text{if } q=0 \\
0 & \text{if } q>0
\end{cases}
\]
Hence the right hand side is filtered by sub-quotients of the modules on the left
hand side and it suffices to bound the degree of these simultaneously.

Then for  each left hand side term an upper bound is  obtained applying 
Proposition \ref{EVW4.10}:
\[
\deg H_{p-i}(\sK(R)) + \max\{ \deg H_0(M), \deg H_1(M) \} + 4\, A(R)
\\
+ i \tilde{A}(R) \, .
\]
Using Proposition \ref{hpR} the first term can be replaced by 
\[
p-i + \deg U + A(R)
\]
and the last term is $\geq 1$, hence
a common upper bound is obtained for $i=p$ and is therefore
\[
\deg U + \max\{ \deg H_0(M), \deg H_1(M) \} + 5\, A(R)
\\
+ p \tilde{A}(R) \, .
\]
The claim then follows using $A(R)\leq  \tilde{A}(R)$
and the  notation we set up.
\end{proof}

\begin{theo}\label{HAM}
Let $G$ be a finite group, let $R$ be the associated ring of connected components. 
Then, for a left graded $R$-module $M$,  the following hold true: 
\begin{itemize}
\item[1.] \begin{equation}\label{hpM}
h_p (M) \leq \max \{ h_0 (M), h_1 (M) \} + (5+p)\tilde{A}(R) +\deg (U) \, , \qquad \forall \, p \geq 0 \, ;
\end{equation}
\item[2.]   ${\rm U}\colon M_{g'}\to M_{{g'}+1}$ is an isomorphism   $\forall \, {g'}\geq \max \{ h_0 (M), h_1 (M) \} + 5A(R)+1$;
\item[3.]
$h_0(R)$ and $h_1(R)$ are both finite.
\end{itemize}
\end{theo}
\begin{proof}
1. We prove that
\begin{equation}\label{star thm 3.21}
\deg (H_i(M)) \leq \deg H_i(\mathcal{K}(M)) \, , \quad {\rm for} \, i=0,1 \, ,
\end{equation}
then the claim follows directly from Lemma \ref{universalcoefficients} (recall that,
according to Notation \ref{notation degM}, $h_p(M)= \deg (H_p(\mathcal{K}(M))$).

Note that $H_0(M) \cong H_0(\mathcal{K}(M))$, hence \eqref{star thm 3.21} holds true when $i=0$.
To prove \eqref{star thm 3.21} for $i=1$ let us consider the final term of the $\mathcal{K}$-complex of $M$:
$$
\mathbb{Z} \langle G^2 \rangle \otimes M[1] \stackrel{d}{\longrightarrow} M \, ,
$$
and let us factor the boundary operator $d$ as follows:
$$
\mathbb{Z} \langle G^2 \rangle \otimes M[1] \stackrel{\alpha}{\longrightarrow} 
R_{>0} \otimes_R M \stackrel{\beta}{\longrightarrow} M \, ,
$$
where $\alpha$ sends $(a,b)\otimes m$ to $\llbracket a, b\rrbracket  \otimes m$, and $\beta$ sends $x\otimes m$ to $xm$.
Note that $\alpha$ is surjective (since, as $R$-module, $R_{>0}$ is generated by the elements
$\llbracket a, b\rrbracket $, for $(a, b) \in G^2$) and that it preserves the degrees. We  have also the following isomorphism:
\begin{equation}\label{** thm 3.21}
H_1(M) \cong \ker (\beta) \, ,
\end{equation}
which follows from the long exact sequence for $\tor^R_i(\_ , M)$ associated to 
the short exact sequence $0\to R_{>0} \to R \to \mathbb{Z} \to 0$.

On the other hand, we  always have the following exact sequence:
$$
\mathbb{Z} \langle G^4 \rangle \otimes M[2] \stackrel{d}{\longrightarrow} 
\ker (\beta \circ \alpha)  {\longrightarrow} H_1(\mathcal{K}(M)) \to 0 \, , 
$$
where $\mathcal{K}(M)_2 \stackrel{d}{\longrightarrow} \mathcal{K}(M)_1$ 
is the boundary operator of the $\mathcal{K}$-complex of $M$. Note that $\ker (\beta \circ \alpha) = \alpha^{-1} (\ker (\beta))$.

Now, let $x\in \ker (\beta) \cong H_1(M)$ be a homogeneous element of degree $\deg (x) > h_1(M)=\deg H_1(\mathcal{K}(M))$.
Since $\alpha$ is surjective, there exists an $y\in \mathbb{Z} \langle G^2 \rangle \otimes M[1]$
such that $x=\alpha (y)$. Note that $y\in \ker (\beta \circ \alpha)$ by construction. 
Moreover $\deg (y) = \deg (x)> h_1(M)$, because $\alpha$ preserves the degrees, therefore 
$y=d(z)$, for some $z\in \mathbb{Z} \langle G^4 \rangle \otimes M[2]$. 
We conclude that $x=\alpha(d(z))$. Since $\alpha \circ d =0$ (as we show below), it follows that $x=0$,
hence $\deg H_1(M) \leq h_1(M)$, i.e. \eqref{star thm 3.21} holds true for $i=1$.

To see that $\alpha \circ d =0$, let $(a_1, b_1, a_2, b_2)\otimes m \in \mathbb{Z} \langle G^4 \rangle \otimes M[2]$.
Then 
$$
d \left( (a_1, b_1, a_2, b_2)\otimes m \right) = (a_2, b_2)\otimes \llbracket a_1^{[a_2, b_2]}, b_1^{[a_2, b_2]}\rrbracket m - 
(a_1, b_1)\otimes \llbracket a_2, b_2\rrbracket m \, ,
$$
and 
\begin{eqnarray*}
&&\alpha \left( (a_2, b_2)\otimes \llbracket a_1^{[a_2, b_2]}, b_1^{[a_2, b_2]}\rrbracket m - 
(a_1, b_1)\otimes \llbracket a_2, b_2\rrbracket m \right) \\
&& = \llbracket a_2, b_2\rrbracket \otimes \llbracket a_1^{[a_2, b_2]}, b_1^{[a_2, b_2]}\rrbracket m - 
\llbracket a_1, b_1\rrbracket \otimes \llbracket a_2, b_2\rrbracket m \\
&& =  \llbracket a_2, b_2\rrbracket \llbracket a_1^{[a_2, b_2]}, b_1^{[a_2, b_2]}\rrbracket \otimes m - 
\llbracket a_1, b_1\rrbracket \llbracket a_2, b_2\rrbracket \otimes m =0 \, .
\end{eqnarray*}
Note that, in the second-last equality we have used equation 2.2. of \cite{Zimmer}.

2.  Lemma \ref{lemma 3.12} implies that  
$U \colon M_{g'} \to M_{{g'}+1}$ is an isomorphism for every ${g'}\geq \delta (M) + A(R) +1$.
The claim follows directly from Lemma \ref{EVW4.9} and \eqref{star thm 3.21}.

3. This claim follows from Proposition \ref{hpR}  and Remark \ref{A(R)finite}.
\end{proof}

\section{Main theorem for $n=0$}

\begin{theo}\label{Ustabilizes}
Let $G$ be a finite group. Let $R$ be the ring of connected components associated to $G$,
and let $M(q)$ be the left $R$-module defined in Example \ref{M(q)}. 
Then the homomorphism
$$
{U} \colon  M(q)_{g'} \to M(q)_{{g'}+1}
$$
is an isomorphism for every  ${g'}\geq (8\tilde{A}(R) + \deg (U))q + \tilde{A}(R) +6A(R) +2$. 
\end{theo}
\begin{proof} 
Set 
\begin{eqnarray*}
A &:=& 8\tilde{A}(R) + \deg (U) \, , \\
B & := & \tilde{A}(R) \, , \\
C &:= &A(R) +1 \, .
\end{eqnarray*}
We prove  that 
\begin{equation}\label{hpMq}
h_p(M(q)) \leq Aq  + Bp+ C \, , \quad \forall \, p, q \geq 0 \, .
\end{equation}
In particular we have that $\max \{h_0(M(q)), h_1(M(q))\} \leq Aq+B+C$ and 
so the claim follows from Theorem \ref{HAM}, 2.

To prove \eqref{hpMq} we proceed by induction on $q$. For $q=0$,  $M(0)=R$
and  Proposition \ref{hpR} yields the inequality $h_p(R) \leq p+A(R)+1$.
Hence \eqref{hpMq} holds true in this case because $1 \leq \tilde{A}(R)=B$.

Let now $q'>0$ and let  us suppose  that \eqref{hpMq} holds true for every $q<q'$ and  $p\geq 0$.
Consider the first three terms of the $\mathcal{K}$-complex of $M(q')$, $(\mathcal{K}(M(q')), d)$:
\begin{equation}\label{hpMq+1}
0 \leftarrow M(q') \xleftarrow{d_1} \ZZ\langle G^2 \rangle \otimes M(q')[1] \xleftarrow{d_2} \ZZ \langle G^4 \rangle \otimes M(q')[2] \, ,
\end{equation}
where, to simplify the discussion below, we have denoted with $d_1$ and $d_2$ the corresponding
boundary operators.
Recall that, for $p>1$, the   term in degree ${g'}$ of $d \colon \mathcal{K}(M(q'))_{p} \to \mathcal{K}(M(q'))_{p-1}$
coincides with $d^1 \colon E^1_{p-1,q'}({g'}) \to E^1_{p-2, q'}({g'})$ (the differential described in Proposition \ref{d^1}).
Therefore, for $p>1$, $E^2_{p-1, q'}({g'})$ is equal to the  term in degree ${g'}$ of $H_{p}(M(q'))$.
While, for $p=1$, 
the term in degree ${g'}$ of $d\colon \mathcal{K}(M(q'))_1 \to M(q')$ 
is the edge map $E^1_{0, q'}({g'}) \to E^\infty_{0, q'}({g'})$ (Proposition \ref{d^1}).

By the inductive  hypothesis we have that
\[
E^2_{j,q'+1-j}({g'}) \cong H_{j+1}(\mathcal K (M({q'+1-j}))_{g'})=0 \, , 
\]
for every $j>1$ and  ${g'}> A(q'+1-j) + B(j+1)+ C$. 
In particular, for every ${g'}>Aq' -A+3B+C$, $E^2_{2, q'-1}({g'})=0$ and so 
there is no differential with target $E^2_{0,q'}({g'})$. Moreover, for  $r>2$, the same is true 
for $E^r_{0,q'}({g'})$, since $E^r_{r,q'-r+1}({g'})$ is a sub-quotient of $E^2_{r,q'-r+1}({g'})$, and $E^2_{r,q'-r+1}({g'})=0$
for every ${g'}>Aq' -A+3B+C$ (note that $B-A<0$).
Hence 
\begin{equation}\label{*** theo 4.1}
E^2_{0,q'}({g'}) =E^\infty_{0,q'}({g'}) \, , \qquad \forall \, {g'}>Aq' -A+3B+C \, .
\end{equation}

On the other hand, again  by the  inductive hypothesis, we have that
$$
E^2_{j,q'-j}({g'}) \cong H_{j+1} (\mathcal K (M(q'-j))_{g'})=0 \, , 
$$
for $j>0$ and ${g'}>A(q'-j) +B(j+1)+C$.
In particular, since $B-A<0$,  
$$
E^2_{j,q'-j}({g'})=0 \, , \quad \forall  j>0 \, , {g'}>  Aq' -A+2B+C \, .
$$
Therefore $E^\infty_{j,q'-j}({g'})=0$ in the same range, and hence
$$
M(q')_{g'}= E^\infty_{0, q'}({g'})= E^2_{0,q'}({g'}) \, , \quad \forall \, {g'}> Aq' -A+3B+C \, ,
$$
where in the last equality we have used \eqref{*** theo 4.1}.

From this we conclude that $d_1$ in \eqref{hpMq+1} induces an isomorphism
$$
M(q')_{g'} \leftarrow {\rm coker}(d_2) =E^2_{0,q'}({g'}) \, , \quad \forall \, {g'}> Aq' -A+3B+C \, .
$$
Equivalently, \eqref{hpMq+1} is exact to the left and in the middle, in the same range, i.e.
$$
h_0 (M(q')) , h_1(M(q'))\leq Aq' -A+3B+C \, .
$$ 
Theorem \ref{HAM}, 1., now implies that 
$$
h_p(M(q')) \leq Aq' -A+3B+C + (5+p)\tilde{A}(R) +\deg (U) = 
Aq' +Bp+C \, .
$$
Hence \eqref{hpMq} holds true for $q'$.
\end{proof}

\section{ The Main Theorem: general case with   $ n > 0$.}
\label{general}
In this section we consider Galois $G$-coverings $C\to C' := C/G$ of compact Riemann surfaces $C'$ of genus $g'$,
which are ramified over $n\geq 0$ marked points and with one tangentially framed orbit.
We show that the homology groups of the corresponding moduli space, for fixed $n\geq 0$, stabilize when $g' >>0$. 

We use, as before, Teichm\"uller theory, therefore the homology groups with coefficients in $\ZZ$ of this moduli space 
(considered as a Deligne-Mumford stack) coincide with the  homology groups
$H_*(\Gamma_{g', 1}^n, \ZZ \langle G^{2g'+n} \rangle)$ (see, e.g. \cite{Behrend}), of $\Gamma_{g', 1}^n$
with values in the $\Gamma_{g', 1}^n$-module $\ZZ \langle G^{2g'+n} \rangle$, 
where $G^{2g'+n}$ is identified with the set of group homomorphisms 
$$
\mu \colon \pi_1(\Sigma^n_{g', 1}\setminus \sB_\Sigma, y_0) \to G
$$
(after the choice of a geometric basis for the fundamental group).

Let $R$ be the ring of connected components defined in Section \ref{ring of connected components}.
For any $q\geq 0$, the $q$-th homology (graded) module is (in this case) defined as 
\begin{equation}\label{Module^n}
M^n(q)= \oplus_{g'\geq 0}M^n(q)_{g'} :=  \oplus_{g'\geq 0} H_q ({ \Gamma^n_{g', 1}}, \ZZ \langle G^{2{g'}+n} \rangle ) \, 
\end{equation} 
(in analogy with Example \ref{M(q)}). 

We define a structure of graded left $R$-module on $M^n(q)$ as follows.
For any $\ell \geq 0$
and $(\tilde{a}_1, \tilde{b}_1, \ldots , \tilde{a}_\ell, \tilde{b}_\ell) \in G^{2\ell}$, the map
\begin{eqnarray*}
\ZZ \langle G^{2{g'} + n } \rangle &\to& \ZZ \langle G^{2\ell + 2{g'} + n} \rangle \\
 ( a_1, b_1, \ldots , a_{g'}, b_{g'}, c_1, \ldots ,c_n ) &\mapsto& 
(\tilde{a}_1, \tilde{b}_1, \ldots , \tilde{a}_\ell, \tilde{b}_\ell, a_1, b_1, \ldots , a_{g'}, b_{g'}, c_1, \ldots , c_n )
\end{eqnarray*}
is equivariant with respect to the inclusion $ \Gamma^n_{g', 1} \hookrightarrow \Gamma^n_{\ell+{g'}, 1}$.
Therefore it induces a homomorphism between homology groups 
$$
H_q ({ \Gamma^n_{g', 1}}, \ZZ \langle G^{2{g'}+n} \rangle ) \to 
H_q ({ \Gamma^n_{\ell+g', 1}}, \ZZ \langle G^{2\ell+2{g'}+n} \rangle ) \, , \quad \forall q\geq 0 \, .
$$
The same argument as in Example \ref{M(q)} shows that this homomorphism depends only on the class 
$\llbracket \tilde{a}_1, \tilde{b}_1, \ldots , \tilde{a}_\ell, \tilde{b}_\ell \rrbracket \in R_\ell$, it will be denoted as follows:
\begin{eqnarray*}
 H_q ({ \Gamma^n_{g', 1}}, \ZZ \langle G^{2{g'}+n} \rangle ) & \to & 
 H_q ({ \Gamma^n_{\ell+{g'}, 1}}, \ZZ \langle G^{2\ell+2{g'}+n} \rangle ) \\
 m &\mapsto& \llbracket \tilde{a}_1, \tilde{b}_1, \ldots , \tilde{a}_\ell, \tilde{b}_\ell \rrbracket  m \, .
\end{eqnarray*}
Therefore $M^n(q)$ is a graded left $R$-module and so we can consider its associated 
$\mathcal{K}$-module, $\mathcal{K}(M^n(q))$ (defined in Definition \ref{defKM}). 
Moreover, the operator $U \colon M^n(q) \to M^n(q)$ is well defined
and has degree $1$.

\begin{rem}
Note that $Mn(q)$ has also a natural right-action by $R$, it is induced by the map  
\begin{eqnarray*}
\ZZ \langle G^{2{g'}+n} \rangle &\to& \ZZ \langle G^{2{g'} +2\ell +n} \rangle \\
(a_1, b_1, \ldots , a_{g'}, b_{g'}, c_1, \ldots ,c_n) &\mapsto& 
( a_1, b_1, \ldots , a_{g'}, b_{g'},\tilde{a}_1^{c^{-1}}, \tilde{b}_1^{c^{-1}}, \ldots , \tilde{a}_\ell^{c^{-1}}, \tilde{b}_\ell^{c^{-1}}, c_1, \ldots , c_n ) \, ,
\end{eqnarray*}
for any $(\tilde{a}_1, \tilde{b}_1, \ldots , \tilde{a}_\ell, \tilde{b}_\ell) \in G^{2\ell}$ with $c:=c_1\cdots c_n$ a shorthand notation.
\end{rem}


We will need the following results about $M^n(0)$, which correspond to Lemma \ref{R>0killsHKR} and
Proposition \ref{hpR}. Note that, for any $g'\geq 0$, $M^n(0)_{g'}$ is a free $\ZZ$-module of rank equal to the number
of $\Gamma^n_{g', 1}$-orbits in $G^{2g'+n}$, which coincides with the number of connected components 
of the moduli space of $G$-coverings of curves of genus $g'$, branched over $n$ points.
It follows from \cite{CLP16} that $A(M^n(0))$ (defined in Notation \ref{notation degM}) is finite, 
i.e. $U \colon M^n(0)_{g'} \to M^n(0)_{g'+1}$
is an isomorphism, for $g'>>0$.

\begin{prop}\label{hpMn}
1. For any $x\in R_{>0}$, the right multiplication by $x$ on  $M^n(0)$ gives a homomorphism of chain complexes
$\mathcal{K}(M^n(0)) \to \mathcal{K}(M^n(0))$, which induces the zero-map in homology.  

2. $A(M^n(0))$ is finite and, for any $p\geq 0$, we have that 
$$
h_p(M^n(0)) \leq p + A(M^n(0)) +1 \, .
$$
\end{prop}
\begin{proof}
1. The proof of this claim is a straightforward adaptation of the one of Lemma \ref{R>0killsHKR}.
Indeed: 
A $\ZZ$-module basis of $\sK (M^n(0)) _{p+1}$ is given by elements of the form
\[
(a_1,b_1,\dots, a_{p+1},b_{p+1}) 
\llbracket g_1,h_1,\dots,g_q,h_q ; c_1, \dots, c_n \rrbracket 
\]
For a pair $(g,h)\in G\times G$ define a map 
$S_{(g,h)}:\sK (M^n(0))_{p+1} \to \sK (M^n(0))_{p+2} $ by
\[
(a_1,b_1,\dots,a_{p+1},b_{p+1})
\llbracket g_1,h_1,\dots,g_q,h_q ; c_1, \dots, c_n \rrbracket \hspace*{4cm}
\]
\[
 \mapsto \quad
(g^{\tau^{-1}},h^{\tau^{-1}},a_1,b_1,\dots,a_{p+1},b_{p+1}) 
\llbracket g_1,h_1,\dots,g_q,h_q ; c_1, \dots, c_n \rrbracket 
\]
where $\tau$ is the product of all handle commutators and all branch
monodromies, namely, 
\[
\tau \quad : = \quad [a_1,b_1]\cdots[a_{p+1},b_{p+1}][g_1,h_1]\cdots[g_q,h_q]c_1\cdots c_n
\quad = \quad \tau_{ab}\tau_{gh} c
\]
with
$\tau_{ab} \quad : = \quad [a_1,b_1]\cdots[a_{p+1},b_{p+1}],$
$\tau_{gh} \quad : = [g_1,h_1]\cdots[g_q,h_q]
$
and 
$c = c_1\cdots c_n$,
and one can check that the product $\tau$ is independent of the choice of  all representatives
of the equivalence class 
$ \llbracket g_1,h_1,\dots,g_q,h_q ; c_1, \dots, c_n \rrbracket  $.

Then it is a standard exercise to compute
\begin{eqnarray*}
& & \big( S_{(g,h)} d + d S_{(g,h)} \big)\bigg(
(a_1,b_1,\dots,a_{p+1},b_{p+1}) 
\llbracket g_1,h_1,\dots,g_q,h_q ; c_1, \dots, c_n \rrbracket \bigg)\\
& = &
(a_1,b_1,\dots,a_{p+1},b_{p+1})
 \llbracket g^{\tau^{-1}\tau_{ab}},h^{\tau^{-1}\tau_{ab}},
g_1,h_1,\dots,g_q,h_q ; c_1, \dots, c_n \rrbracket \\
& = &
(a_1,b_1,\dots,a_{p+1},b_{p+1})
 \llbracket g^{(\tau_{gh}c)^{-1}
},h^{(\tau_{gh}c)^{-1}
},
g_1,h_1,\dots,g_q,h_q ; c_1, \dots, c_n \rrbracket \\
& = &
(a_1,b_1,\dots,a_{p+1},b_{p+1})
 \llbracket 
g_1,h_1,\dots,g_q,h_q, 
g^{c^{-1}}, h^{c^{-1}}
 ; c_1, \dots, c_n \rrbracket\\
& = &
(a_1,b_1,\dots,a_{p+1},b_{p+1}) \llbracket g_1,h_1,\dots,g_q,h_q  ; c_1, \dots, c_n \rrbracket \cdot
\llbracket g, h \rrbracket
\end{eqnarray*}
The first equality is a straightforward  application of the definitions of $S_{(g,h)}$ and $d$, followed by a cancellation
which leaves only out the first term of the sum corresponding to $dS_{(g,h)}$, 
which is exactly the term appearing  in the second row.\\
The third  equality is due to a sequence of Zimmermann moves (2.2, page 250  in  \cite{Zimmer}), which allows to swap adjacent
pairs of group elements, provided that one of them is conjugated by the commutator of the other.\\
The last equality is due to the definition of the right module structure.
\\
Thus $S_{(g,h)}$ provides a chain homotopy between the zero map and 
the multiplication by $\llbracket g, h \rrbracket$ on the right.
The claim then follows, since $R_{>0}$ is generated by such elements.

2. The fact that $A(M^n(0)) <0$ follows from \cite{CLP16}, as explained above.
The same argument in the proof of Proposition \ref{hpR} implies that $h_p(M^n(0)) \leq p + A(M^n(0)) +1$,
for any $p\geq 0$.
\end{proof}

Then we have the following extension of Theorem \ref{Ustabilizes}.

\begin{theo}\label{Ustabilizes^n}
Let $n\geq 0$ be a fixed natural number,  let $M^n(q)$ be the graded left $R$-module 
defined in \eqref{Module^n}. Then, for any ${g'}\geq (8\tilde{A}(R) + \deg (U))q + \tilde{A}(R) +5A(R) +
A(M^n(0)) +2$,
the homomorphism
$$
{U} \colon  M^n(q)_{g'} \to M^n(q)_{{g'}+1}
$$
is an isomorphism, where $\tilde{A}(R)$ and ${A}(R)$ are defined in Notation \ref{notation degM}. 
\end{theo}

The proof of this theorem follows the same steps as those of Theorem \ref{Ustabilizes}, with the following changes. 

A tethered chain in $\Sigma^n_{g', 1}$ is defined as in Section \ref{tch complex}, 
with the additional requirement that it doesn't contain any marked point. 
The complex of tethered chains on  $\Sigma^n_{g', 1}$, $\tch^n_{g', 1}$,
is defined accordingly and it is $\frac{1}{2}(g'-3)$-connected. 
To see this, consider the surface $\Sigma^0_{g', 1+n}$ obtained from $\Sigma^n_{g', 1}$ by removing 
$n$ open disks centered at the marked points and pairwise disjoint. The inclusion 
$\Sigma^0_{g', 1+n} \subseteq \Sigma^n_{g', 1}$ yields an identification $\tch^n_{g', 1} \to \tch_{g', 1 +n}$
and so the claim follows from \cite[Proposition 5.5]{h-v}.

The  complex of $G$-marked tethered chains in $\Sigma^n_{g', 1}$ is defined as in Definition \ref{G-marked tch},
it is the simplicial complex 
\[
X^n_{g'} := { \tch^n_{g', 1}} \times G^{2{g'}+n} \, .
\]
It will be considered with the diagonal left action of $\Gamma^n_{g', 1}$. 

Replacing (in Section \ref{the spectral sequences}) the chain complex of $X_{g'}$ with that of $X_{g'}^n$,
we obtain two spectral sequences, $\tilde{E}_{p,q}^r ({g'})$ and $E^r_{p,q} ({g'})$,
both converging to  the  homology groups $H_{p+q}(\Gamma^n_{g', 1}, \ZZ \langle G^{2{g'}+n}\rangle)$, 
if $p+q \leq \frac{1}{2}({g'}-3)$. Proposition \ref{d^1} holds true also in this case, 
where $H_q({ \Gamma_{g'', 1}}, \ZZ \langle G^{2{g''}}\rangle)$ is replaced by 
$H_q({ \Gamma^n_{g'', 1}}, \ZZ \langle G^{2{g''}+n}\rangle)$, for $g''=g', g'-p, g'-p-1$. 
It follows from this that, for $p>1$, the differential of the second spectral sequence 
$d^1 \colon E^1_{p-1,q}({g'}) \to E^1_{p-2, q}({g'})$ coincides with the   term in degree ${g'}$ of 
the $\mathcal{K}$-complex associated to $M^n(q)$, $d \colon \mathcal{K}(M^n(q))_{p} \to \mathcal{K}(M^n(q))_{p-1}$.
While, for $p=1$, 
the term in degree ${g'}$ of $d\colon \mathcal{K}(M^n(q))_1 \to M^n(q)$ 
is the edge map $E^1_{0, q}({g'}) \to E^\infty_{0, q}({g'})$ (compare with the discussion at the end of Example \ref{M(q)}).

Set now
\begin{eqnarray*}
A &:=& 8\tilde{A}(R) + \deg (U) \, , \\
B & := & \tilde{A}(R) \, , \\
C^n &:= &A(M^n(0)) +1 \, .
\end{eqnarray*}
The same arguments as in the proof of \eqref{hpMq} imply that
\begin{equation}\label{hpMq_n}
h_p(M^n(q)) \leq Aq  + Bp+ C^n \, , \quad \forall \, p, q \geq 0 \, .
\end{equation}
In particular we have that $\max \{h_0(M^n(q)), h_1(M^n(q))\} \leq Aq+B+C^n$ and 
so the claim follows from Theorem \ref{HAM}, 2.

\qed

The conclusion is then  that, since Theorem \ref{Ustabilizes^n} generalizes 
Theorem \ref{Ustabilizes} for the  more general case where $ n > 0$, 
 Theorem \ref{main} is then proven in full generality.

\bigskip

{\bf Acknowledgements.} 
We are grateful to various  sources for funding this research:
the  ERC Advanced Grant 340258 TADMICAMT, 2014-2020;
the National Project PRIN 2022 ``Geometry of Algebraic Structures: Moduli, Invariants, Deformations";
the research group GNSAGA of INDAM and the FRA  of the University of Trieste.

 Thanks to a referee for suggestions and for convincing us to put pictures illustrating our contructions.


\medskip
\noindent {\bf Authors' Addresses:}\\

\noindent Fabrizio Catanese,\\ 
Michael L\"onne, \\
 Lehrstuhl Mathematik VIII,  Mathematisches
Institut der Universit\"at
Bayreuth\\ NW II,  Universit\"atsstr. 30,  95447 Bayreuth (Germany).\\

\noindent Fabio Perroni, \\
Dipartimento di Matematica, Informatica e Geoscienze, \\
Universit\`a degli Studi di Trieste,
via Valerio 12/1, 34127 Trieste, Italy.\\

\noindent           email:                 
        Fabrizio.Catanese@uni-bayreuth.de,
Michael.Loenne@uni-bayreuth.de, 
fperroni@units.it

\end{document}